\documentclass{amsart}
\usepackage{amssymb,amsfonts,latexsym}
\usepackage{graphics,verbatim}
\usepackage{graphicx}

\newtheorem{thm}{Theorem}[subsection]

\newtheorem{prop}[thm]{Proposition}
\theoremstyle{definition}

\theoremstyle{remark}

\theoremstyle{plain}
\newtheorem{theorem} {Theorem}[section]

\newtheorem{lemma}[theorem]{Lemma}

\theoremstyle{definition}

\theoremstyle{remark}
\newtheorem{remark}[theorem]{Remark}

\newcommand{\be} {\begin{equation}}
\newcommand{\ee} {\end{equation}}
\newcommand{\bea} {\begin{eqnarray}}
\newcommand{\eea} {\end{eqnarray}}
\newcommand{\Bea} {\begin{eqnarray*}}
\newcommand{\Eea} {\end{eqnarray*}}
\newcommand{\pa} {\partial}

\newcommand{\al} {\alpha}

\newcommand{\de} {\delta}

\newcommand{\Om} {\Omega}
\newcommand{\om} {\omega}
\newcommand{\De} {\Delta}

\newcommand{\no} {\nonumber}

\newcommand{\var} {\varepsilon}

\newcommand{\R}{\mathbb R}
\newcommand{\N}{\mathbb N}
\catcode`\@=11

\makeatletter
\@addtoreset{equation}{section}
\makeatother

\begin{document}

\title[]{New entire positive solution for the nonlinear Schr\"{o}dinger equation: Coexistence of fronts and bumps}

\author{Sanjiban Santra, Juncheng Wei }

\address{S. Santra, School of Mathematics and Statistics, The University of Sydney, NSW 2006, Australia.}
\email{sanjiban.santra@sydney.edu.au}
\address{J. Wei, Department of Mathematics, The Chinese University of Hong Kong, Shatin, Hong Kong.}
\email{wei@math.cuhk.edu.hk}



\begin{abstract}
In this paper we construct a new kind of positive solutions of $$\De
u-u+u^{p}=0  \text{ on } \R^2$$ when $p> 2.$ These solutions have the following asymptotic behavior
$$\ \displaystyle{u(x,z)\sim \om(x-f(z))+
\sum_{i=1}^{\infty}\om_{0}((x, z)-\xi_i\vec{e}_{1})}$$
 as $L\rightarrow +\infty$ where $\om$ is a unique positive homoclinic
solution of $\om''-\om+\om^{p}=0$ in $\R$ ; $\om_{0}$ is the two
dimensional positive solution and $\vec{e}_{1}= (1, 0)$ and
$\xi_{j}$ are points such that $\xi_{j}= jL+ \mathcal{O}(1)$ for all
$j\geq 1.$ This represents a first result on the {\em coexistence}
of fronts and bumps. Geometrically, our new solutions correspond to
{\em triunduloid} in the theory of CMC surface.
\end{abstract}
\maketitle
\section{Introduction}
\subsection{Entire Solutions}
Positive entire solutions of \be\label{1}\De u-u+u^{p}=0  \text{ on } \R^N\ee where $1<p< (\frac{N+2}{N-2})_{+},$
 vanishing at infinity have been studied in many contexts. This class of
problems arises in plasma and condensed-matter physics. For
example, if one simulates the interaction-effect among many particles
by introducing a nonlinear term, we obtain a nonlinear
Schr\"odinger equation,
$$-i \frac{\pa \psi}{\pa t}= \De _{x} \psi- \psi+|\psi|^{p-1} \psi$$ where $i$ is an imaginary unit and
$p>1.$ Making an {\it Ansatz}
$$\psi(x,t)=exp(-i t) u(x)$$
one finds that a stationary wave $u$ satisfies (\ref{1}) (\cite{K}). \\

In recent years, much attention has been devoted to the study of existence and multiplicity of  positive solutions of
\be\no \var^2\De u- V(x)u+u^{p}=0;~~  u\in H^{1}(\R^N)\ee
as $\var\rightarrow 0.$ Floer--Weinstien \cite{FW}  constructed
single spike solutions concentrating around any given non-degenerate critical point
of the potential $V$ in  $\R$ provided $\inf_{\R} V>0$, using Lyapunov-Schmidt reduction. This was later extended by
Oh \cite{OG}, \cite{OG1} for the higher dimensional case. \\
Spike layered solutions (solutions concentrating in zero dimensional
sets) in bounded domain $\Om$ with Dirichlet and Neumann boundary
condition have been studied in recent years by many authors. See for
example, Ni-Wei \cite{NW}, Lin--Ni--Wei\cite{LWW}, and the review
articles by Ni \cite{Nisurvey} and Wei \cite{Wei}.
Higher-dimensional concentration is later on studied by
Malchiodi-Montenegro \cite{MM1}-\cite{MM2} in the Neumann case and
by del Pino- Kowalczyk-Wei \cite{DKW1} in $\R^2.$ \\

In this paper, we focus on positive solutions to (\ref{1}). The
solution to (\ref{1}) that is decaying at $\infty$  is
well-understood: all such solutions are radially symmetric around
some point (Gidas-Ni-Nirenberg \cite{GNN}), and are unique modulo
translations (Kwong \cite{K}).  Though solutions of (\ref{1}) are
bounded (since $p<(\frac{N+2}{N-2})_{+}$), not much is known about
the solutions which does not decay at infinity \cite{PQS}. One
obvious solution of such kind is the following: if we consider a
solution $W_{N-1}$ of (\ref{1}) in $\R^{N-1}$ which decays at
infinity, it  induces a solution  in $\R^{N}$ which depends on $N-1$
variables and decays at infinity except for one direction. In the
case $N=2,$ consider solutions $u(x, z)$ to problem (\ref{1}) which
are even in $z$ and vanish at $|x|\rightarrow \infty,$ \be u(x,
z)=u(x, -z) ~~~\forall (x,z)\in \R^2\ee  and \be
\lim_{|x|\rightarrow \infty} u(x, z)=0 ~~~\forall z\in \R.\ee In
\cite{END}, Dancer used local bifurcation arguments to obtain a
class of solutions which constitute a one parameter family of
solutions that are periodic in the $z$ variable and originate from
$\om$,  where $\om$ is the unique positive solution of \be\label{a1}
\om''-\om+\om^{p}=0, \om>0, \om (x)= \om (-x) \text{ in } \R; \,\,  \om\in H^{1}(\R).
\ee
These solutions are called {\em Dancer's solutions}.
They can be parameterized by a small parameter $\delta >0$ and
asymptotically
\begin{equation}
\omega_\delta (x, z)= \omega (x) + \delta \omega^{\frac{p+1}{2}}  (x) \cos (\sqrt{\lambda_1} z) + \mathcal{O}(e^{-|x|}).
\end{equation}

In a seminal paper \cite{M},   Malchiodi constructed a new kind of solutions  with three rays of bumps. More precisely, the solutions constructed in \cite{M} have the form
\begin{equation}
u(x, z) \approx
 \sum_{j=1}^3 \sum_{i=1}^{+\infty} \omega_0 ( (x,z)- i L \vec{l}_j)
\end{equation}
where $ \vec{l}_j, j=1,2,3$ are three unit vectors satisfying some balancing conditions (${\bf Y}$-shaped solutions, see Figure 1). Here $\omega_0$ is the unique solution to the two dimensional entire problem
\begin{equation}
\label{ground2}
\left\{ \begin{array}{l}
\Delta \omega_0-\omega_0 + \omega_0^p=0, \omega_0>0, \ \\
\omega_0 \in H^1 (\R^2). \\
\end{array}
\right.
\end{equation}

\begin{figure}[h] \label{Figure 1}
\centerline{\includegraphics[width=8cm]{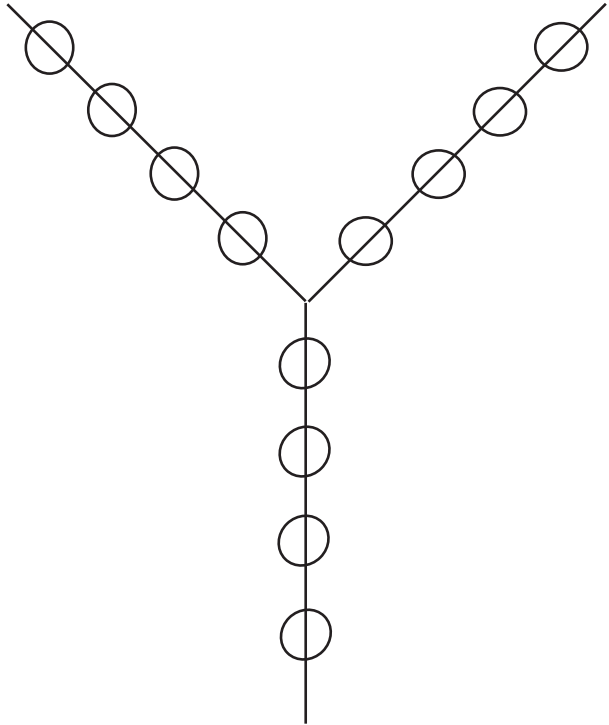}}
\caption{ Multi-bump  solutions with $Y$ shape.}
\end{figure}

On the other hand, in \cite{DKPW1}, del Pino, Kowalczyk, Pacard and Wei constructed another new kind of multi-front solutions using Dancer's solutions and Toda system. (These are solutions with even number of ends. See Figure 2.)  More precisely, the solutions constructed in \cite{DKPW1} have the form
\begin{equation}
u(x, z)\approx \sum_{j=1}^K w_{\delta_j} (x-f_j (z), z)
\end{equation}
where $f_1<f_2 <...< f_K$ satisfies the following Toda system
\begin{equation}
c_0 f_j^{''}= e^{f_{j-1}-f_j}- e^{f_j-f_{j+1}}, f_0=-\infty, f_{K+1}=+\infty, c_0>0.
\end{equation}

\begin{figure}[h] \label{Figure 2}
\centerline{\includegraphics[width=8cm]{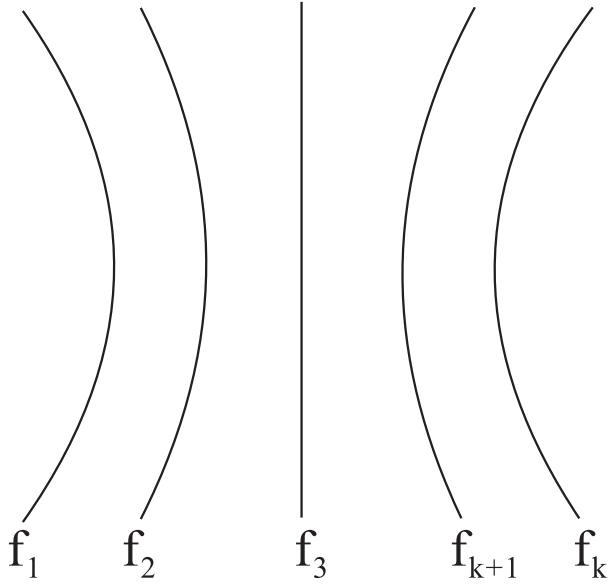}}
\caption{Multi-front  solutions with even-ends.}
\end{figure}

From now on, we call the one-dimensional solution $\omega$ as a ``front'' solution and the two-dimensional solution $\omega_0$ as a ``bump'' solution. Thus results of \cite{DKPW1} and \cite{M} establishes the existence of multi-front and multi-bump solutions respectively.

\subsection{Main Results} In this paper we consider the nonlinear Schr\"odinger equation
\be\label{2}\De u-u+u_{+}^{p}=0  \text{ in } \R^2\ee where $p> 2$ and $u_{\pm}=\max\{\pm u, 0\}$.  Our aim is to  construct solutions with {\em both fronts} and {\em bumps.} More precisely we look for
positive solutions of the form
\begin{equation}\label{w2} u_{\sharp}(x,z)= \om (x-f(z))+ \sum_{i=1}^{\infty}\om_{0}((x, z)-\xi_i\vec{e}_{1})\end{equation}
for suitable large $L>0$ and $\xi_i$'s are such that $\xi_{1}- f(0)= L$ and
$$\xi_1<\xi_2< \cdots <\xi_i< \cdots$$ and satisfy \be\label{w5}  \xi_{j}= j L+ \mathcal{O}(1) \ee for all $j\geq 1;$ $\om$ is the unique even  solution to (\ref{a1}),  $\om_{0}$ is the unique positive solution of (\ref{ground2}) and  $\vec{e}_{1}=(1, 0).$ Along the line of the proof we will replace $u_{+}$ by $u.$

Because of the interaction between the front and the bumps, we are led to considering the following second order ODE:
 \begin{equation}
  \label{a2}
\left\{\begin{aligned}
       f''(z) &= \Psi_L (f, z)  &&\text{in } \R\\
      f(0) &= 0, \ \       f'(0) = 0,
    \end{aligned}
  \right.
\end{equation}
where $ \Psi_L (f, z)$ is a function measuring the interactions between bumps and fronts which will be defined in Section 2. Asymptotically $ \Psi_L  (f,z) \sim  ((f-L)^2+z^2)^{-\frac{1}{2}} e^{-\sqrt{(f-L)^2 +z^2} }$. Let $ \alpha= \int_0^{+\infty} \Psi (\sqrt{L^2+z^2}) dz$.

The following is the main result of this paper.

\begin{theorem}\label{1.1} Let $N=2$. For $p>2$ and sufficiently large $L>0$, (\ref{2}) admits a  one parameter family of positive solutions satisfying
\begin{equation}
  \label{a4}
\left\{\begin{aligned} u_{L}(x, z)&= u_{L}(x, -z) &&\text{for all } (x,z)\in \R^2 \\ u_{L}(x,z)&= \bigg(\om_{\delta} (x-f(z)-h_L(z),z)+
 &&\sum_{i=1}^{\infty}\om_{0}((x, z)-\xi_i\vec{e}_{1})\bigg)(1+o_{L}(1))\end{aligned}
  \right.
\end{equation}
where $\delta=\delta_L$ is a small constant, $\omega_\delta$ is the
Dancer's solution, $f$ is the unique solution of (\ref{a2}),
$\xi_{j}$ satisfy (\ref{w5}) and $o_{L}(1)\rightarrow 0$ as
$L\rightarrow +\infty$, and the function $\|h_{L}\|_{C^{2,
\mu}_{\theta}(\R)\oplus \mathcal{E}}\leq C \alpha^{1+\gamma} $ for
some constant $\theta>0, \gamma>0.$ ( $\mathcal{E}$ will be defined
at Section 2.)  Moreover, the solution has three ends.
\end{theorem}

Figure 3 shows graphically how the solution constructed in Theorem
\ref{1.1} looks like triunduloid type I. (This corresponds end-to-end gluing construction. See discussions at the end.)

\begin{figure}[h] \label{Figure 3}
\centerline{\includegraphics[width=8cm]{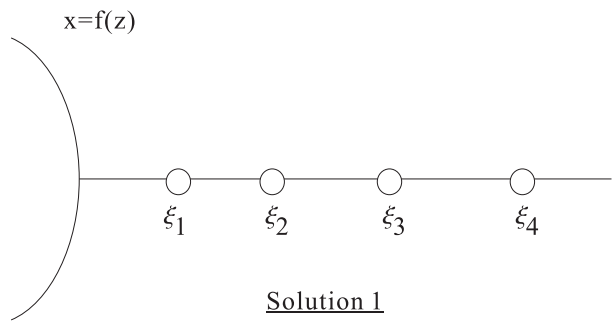}} \caption{
Triunduloid Type I}
\end{figure}

A modification of our technique can be used to construct  the
following two new types of solutions: the first one is a combination
of positive front and infinitely many negative bumps--we call it
Solution 2 (triunduloid type II). The second one is a combination of two fronts
and one bump (or finitely many bumps)--we call it Solution 3.

\vskip 0.3in
\begin{figure}[h]\label{Figure 4}
\centerline{\includegraphics[width=8cm]{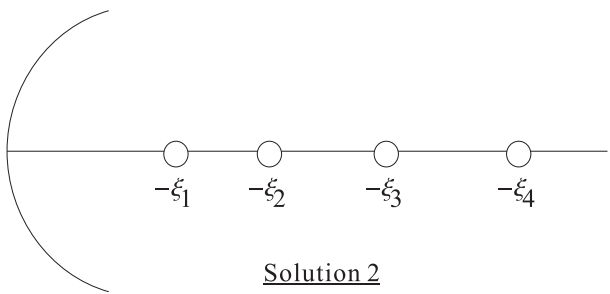}} \caption{
Triunduloid Type II}
\end{figure}
\vskip 0.3in
\begin{figure}[h]\label{Figure 5}
\centerline{\includegraphics[width=8cm]{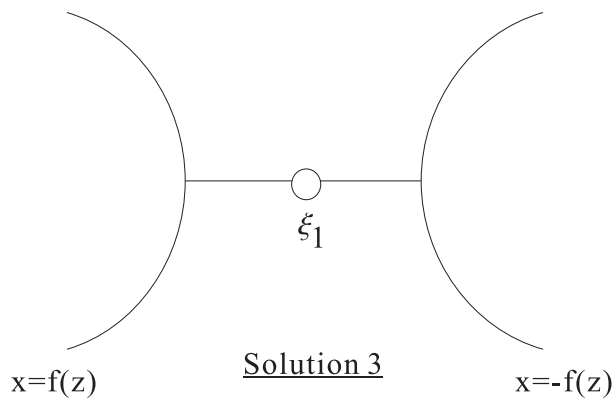}}
\caption{End to end gluing construction}
\end{figure}

In this paper we will only discuss the proofs of Solution 1. The modifications needed for Solution 2 and Solution 3 will be explained at the last section.

Theorem \ref{1.1} implies that we can construct solutions which does
not decay along the $x-$ axis but decay everywhere else. Though Theorem
\ref{1.1} is a purely PDE result, this result has an analogy in the
theory of constant mean curvature (CMC) surface in $\R^3$ which we
shall describe below.

\subsection{Relation with CMC Theory}
  CMC surfaces in $\R^3$ are an equilibria for the area functional subjected to an
enclosed volume constraint. To explain mathematically, suppose an oriented surface $\mathcal{S}$ is embedded in a manifold $M$ and let us denote $\nu_{}$ be the normal field compatible with the orientation. Then for any function $z$  which is smooth small function we define a {\em perturbed surface} $\mathcal{S}_{z}$ as the normal graph of the function of $z$ over $\mathcal{S}.$ Namely $\mathcal{S}_{z}$ is parameterized as
$$p\in \mathcal{S}\mapsto exp(w(p)\nu(p))$$
where $exp$ is the exponential map in $(M, g).$
Decompose $z$ into the positive part  and the negative part of $z$ as $z=z^{+}- z^{-}$ and define the set
$$B_{z^{\pm}}:=\{exp_{p}{t\nu(p)}: \pm t\in (0, z^{\pm}(p))\}.$$
Then the $m$-th volume functional
$$\mathcal{A}(z)= \int_{S_{z^{}}} d vol_{\mathcal{S}_{z}}$$
and its first and second variations at $z=0$ are
$$D \mathcal{A}(0)(v)=\int_{\mathcal{S}}\text{H} v~ d vol_{\mathcal{S}_{}}$$
$$D^2\mathcal{A}(0)(v, v )=\int_{\mathcal{S}}(|\nabla_{g} v|^2- (\kappa_{1}^2+ \kappa_{2}^2+ \cdots \kappa_{m}^2)v^2- Ric(\nu, \nu)v^2+ \text{H}^2 v^2 ) d vol_{\mathcal{S}_{}}$$
where $\kappa_{i}$ are the principal curvatures of $\mathcal{S}$ , $Ric$ denotes the Ricci tensor on $(M, g)$ and $\text{H}$ is the mean curvature function and depends on $\mathcal{S}.$ Also note that the critical points of $\mathcal{A}$ are precisely surfaces of mean curvature zero and usually referred to as minimal surfaces.
Moreover, define $(m+1)$ th volume functional
$$\mathcal{V}(z) :=\int_{B_{z^{+}}} d vol_{M} - \int_{B_{z^{-}}} d vol_{ M}$$
where volumes are counted positively when $w > 0$ and negatively when $w < 0$. The first
variation of $\mathcal{V}$ is given by
$$D \mathcal{V}(0)(v) = \int_{\mathcal{S}}
v dvol_{\mathcal{S}}$$
and its second variation is given by
$$D^2\mathcal{V}(0)(v, v) = - \int_{\mathcal{S}}\text{H} v^2 dvol_{\mathcal{S}}.$$
Define the {\em shape} operator as \be\no |A|^2= \sum_{i=1}^{m}
\kappa_{i}^2.\ee We see that critical points of the functional
$\mathcal{A}$ with respect to some volume constraint $\mathcal{V}$ =
constant have constant mean curvature. Here the mean curvature
appears as a multiple of the Lagrange multiplier associated to the
constraint (and hence it is constant). The surfaces with constant
mean curvature equal to $\text{H}=\lambda$ are critical points of
$\mathcal{W}(\mathcal{S}) := \mathcal{A}(\mathcal{S}) + \lambda
\mathcal{V}(\mathcal{S}).$ The quadratic form can be written as
$$D^2\mathcal{W}(0)(v, v)=- \int_{\mathcal{S}} v \mathcal{J}_{\mathcal{S}}  v ~dvol_{\mathcal{S}}$$
where the {\em Jacobi} operator is given by
\be\label{bh1} \mathcal{J}_{\mathcal{S}}=\De_{\mathcal{S}}+ |A|^2+ Ric_{g}(\nu, \nu).\ee

For CMC surfaces the sign of $\text{H}$ and its value can be changed
by a reversal of orientation and homothety respectively  and as a
result we can normalize the surface such that $\text{H}\equiv1.$ CMC
interfaces arise in many physical  and variational problems. Over
the past two decades there is a great deal of progress in
understanding complete CMC and their {\em moduli} spaces. Moduli  is
a notion to identify invariant surfaces. In order to study the
structure of moduli spaces one needs to study the properties of
(\ref{bh1}). The reflection technique of Alexandrov  \cite{A} shows
that spheres is the only compact embedded CMC surface of {\em
finite} topology. These are surfaces homeomorphic to a compact
surface $\mathcal{S}$ of genus $g$ with a finite number of points
removed from it say $m$. The neighborhood of each of these punctures
are called {\em ends}. Mathematically, we define the ends $e_{j}$ of
an embedded surface $\mathcal{S}$ in $\R^3$ with finite topology to
be a non-compact connected components of the surface near infinity
$$\mathcal{S}\cap (\R^{3}\setminus B_{R_{0}}(0))= \displaystyle{\cup_{j=1}^{m}}e_{j}$$ where $B_{R_{0}}(0)$ denotes a ball of radius $R_{0}$ (is chosen sufficiently large so that $m$ is constant for all $R> R_{0}$). Note that sphere is a zero end surface.

The theory of properly embedded  CMC surfaces, was classified by Delaunay \cite{CD}. These are  rotationally symmetric CMC
surfaces, called {\em unduloids} (having genus zero and two ends). To describe these, consider the cylindrical
graph \be \label{g0} (t, \theta)\mapsto (h(t)\cos \theta, h(t)\sin \theta, t).\ee The CMC graph is an ordinary differential
equation given by,
\begin{equation}
  \label{cyl}
\left\{\begin{aligned}
    h_{tt}-\frac{1}{h}(1+ h_{t}^{2})+ (1+ h_{t}^{2})^{\frac{3}{2}}&=0\\
      \min_{t}h(t)&= \var.
    \end{aligned}
  \right.
\end{equation}
Moreover, all the positive solutions of (\ref{cyl}) are periodic and may be distinguished by their minimum value $\var\in (0, 1],$ which is more often
referred to as the Delaunay parameter of the surface $D_{\tau}$ where $\tau= 2\var- \var^2$. Moreover, when $\tau=1$, $D_{1}$ is a cylinder of radius $1$ and as $\tau
\downarrow 0,$ $D_{\tau}$ converges to an infinite array of mutually tangent spheres of radius $2$ with centers
along the $z$ axis. The family $D_{\tau}$ interpolates between two extremes and $\var$ measures the size of the neck region. Moreover,
using a parameterization (\ref{g0}) and \be t=k(s), h(t)= \tau e^{\sigma(s)},\ee
we obtain the Jacobi operator for the surface $D_{\tau}$ is given by
\be\label{bh2} \mathcal{J}_{D}=\frac{1}{2\tau^2 e^{2 \sigma}}(\pa_{s}^2+ \pa_{\theta}^2+ \tau^2 \cosh 2\sigma)\ee
where $\sigma''+ \frac{\tau^2}{2}\sinh 2 \sigma=0$ and $k'=\frac{\tau^2}{2}(e^{2\sigma}+1).$

These surfaces are periodic and interpolate between the unit cylinder and the singular surfaces formed by a string of
spheres of radius 2, each tangent to the next along a fixed axis. In particular, Delaunay established that every CMC surface of revolution is
necessarily one of these ``Delaunay surfaces''.
Kapouleas \cite{Kap} constructed numerous examples of
complete embedded  CMC surface in $\R^3$ (with genus $g\geq2$ and ends $k\geq 3$) by gluing Delaunay
surfaces onto spheres. In fact he produced CMC surfaces using suitably {\em balanced} simplicial graphs where the $k $ edges are rays tending to infinity. By balancing condition we mean that the force vectors associated with each edge cancel at each vertex.
In fact balancing condition combined with spherical trigonometry plays an important role in classifying CMC surfaces with three ends.
A more flexible gluing techniques was used by Mazzeo and Pacard in \cite{MP} to explore  moduli surface theory which involves
several boundary value problems and then matching the boundary values across the interface.\\
A CMC surface $\mathcal{S}$ of finite topology is {\em Alexandrov-embedded}; if $\mathcal{S}$ is properly immersed,
and if each end  of $\mathcal{S}$ is embedded; there exist a compact manifold $M$ with boundary of dimension three
 and a proper immersion $F: M \setminus\{q_{1}, q_{2}, \cdots, q_{m}\}\rightarrow \R^3$ such that
 $F\mid_{\pa M \setminus\{q_{1}, q_{2}, \cdots, q_{m}\}}$ parameterizes $M.$ Moreover, the mean curvature
normal of $\mathcal{S}$ points into $M.$\\
Then we define {\em triunduloid} as an Alexandrov embedded CMC
surface  having zero genus and three ends. Triunduloids are a basic
building block  for Alexandrov embedded CMC surface with any number
of ends.  Nonexistence of one end Alexandrov embedded CMC surface
was proved by Meeks \cite{Me}.
 Kapouleas \cite{Kap}, G-Brauckmann \cite{G} and Mazzeo-Pacard \cite{MP}  established existence of  triunduloid
with small necksize or high symmetry. In fact G-Brauckmann \cite{G}  used conjugate surface theory construction to obtain families of symmetric
embedded complete CMC surfaces. The geometry of moduli space plays an very important role for the understanding of the structure of CMC's.

The  main aim of this paper is to prove existence of {\em
triunduloid } type of solution for (\ref{1.1}) in $\R^2$ i.e. a
solution having three ends.  Solutions having even number of ends have been
shown to exist in a recent paper of del Pino, Kowalczyk, Pacard and
Wei, see \cite{DKPW1}. \textbf{Y} shaped solutions of (\ref{1}) in
$\R^3$ were constructed
by Malchiodi \cite{M}. Hence Theorem \ref{1.1} proves that the moduli  space $\mathcal{M}_{3}(\R^2)$ of all $3-end$ solutions is nonempty.\\

Geometrically, solutions constructed in Theorem \ref{1.1} correspond to the so-called {\em end-to-end} gluing in CMC. (We are indebted to Prof. F. Pacard for this connection.)  The end-to-end gluing in CMC corresponds to adding a handle to a multi-end CMC surfaces. The procedure has been used in the thesis of J. Ratzkin \cite{ratzin1}. (A similar construction has been done for the construction of positive metrics with constant positive scalar curvature \cite{ratzin2}.) For nonlinear Schrodinger equation, adding a handle means adding a half-ray solution with infinitely many bumps. The solution in Theorem \ref{1.1} represents  first step in adding a handle. We believe that with more work it is possible to add handles to the even number ends solutions constructed in \cite{DKPW1}.

Finally we should also mention that in a recent paper \cite{MPW}, Musso, Pacard and Wei have constructed {\em nonradial finite-energy sign-changing solutions}, using geometric analogue constructions of Kapouleas \cite{Kap}.

\subsection{Main ideas of proof.}
We sketch the main ideas of the proofs of Theorem \ref{1.1}. The solutions we construct have the form
\begin{equation}
  \label{a4-200}
u(x, z)\sim \om_{\delta} (x-f(z),z)+
 \sum_{i=1}^{\infty}\om_{0}((x, z)-\xi_i\vec{e}_{1}).
\end{equation}

 There are three main parts of the proof: firstly, we add a half-line of bumps (corresponds to $ \sum_{i=1}^{\infty}\om_{0}((x, z)-\xi_i\vec{e}_{1})$).  For this part we use the idea of Malchiodi \cite{M}. Namely we need to use Dancer's solutions with large periods and analyze the interactions using Toeplitz matrix.  Secondly, we have a front solution  (corresponds to $ \om_\delta (x-f(z), z)$). This is a two-end solution and we follow the analysis by del Pino, Kowalczyk, Pacard and Wei \cite{DKPW1}. The third part deals with the {\em interaction part}. Because of the exponentially decaying tails of both $\om_\delta$ and $\om_0$, the dominating force is given by the interaction between the first bump and the front only. We have to compute the corresponding ODE which ultimately determines the curve $f(z)$.  In all these three parts, we will make use of the {\em infinite-dimensional Liapunov-Schmidt reduction method}. For this method, we refer to \cite{DKPW1}, \cite{DKPW2},  \cite{DKW1}, \cite{DKW2}.

\section{The Exponential equation, Toeplitz matrix and it linearisation}

\subsection{The differential equation involving $f$}\label{1.3}

In this paper the second order ODE (\ref{a2}) plays an important role. We shall study the properties of this ODE and identify the scaling parameter.

First let us define the function $ \Psi$: let $ \omega$ be the one-dimensional solution and $\omega_0$ be the two-dimensional solution. $\Psi$ measures the interactions between $\omega$ and $\omega_0$ and is defined by
\begin{equation}
\Psi_L (f, z)= p \int_{\R} \omega^{p-1} (x) \omega_x (x) \omega_0 ( \sqrt{ (x+f)^2+z^2}) dx
\end{equation}
Asymptotically
\begin{equation}
\Psi_L (f,z)\sim  (f^2+z^2)^{-\frac{1}{2}} e^{-\sqrt{f^2+z^2}}.
\end{equation}
We also note that
\begin{equation}
\frac{\partial \Psi_L (f,z)}{\partial f} <0, \frac{\partial \Psi_L (f,z)}{\partial z} <0.
\end{equation}

Let $L>>1$ be a fixed large number. We choose the following small parameter
\begin{equation}
\label{alphachoice}
 \al=e^{-\frac{L}{\sqrt{2}}}
\end{equation}
then $\al \rightarrow 0$ as $L \rightarrow \infty.$

For any $0\leq \mu<1$ we define $C^{l, \mu}_{\theta}(\R)$ to be the space of all real-valued functions where $$\|f\|_{C_{\theta}^{l,\mu}(\R)}= \|(\cosh z)^{\theta}f \|_{C^{l, \mu}(\R)}<+\infty.$$ We will fix $\mu$ later. Since  $f''\geq 0$, $f'$ is an increasing function. Also note as $f$ is even, it is enough to study the  behavior of $f$ when $z>0.$ After a translation, (\ref{a2}) becomes
\begin{equation}
  \label{a2s}
\left\{\begin{aligned}
      f''(z) &= \Psi_L (f, z) &&\text{in } \R\\
      f(0) &= L && \\
      f'(0) &= 0.
    \end{aligned}
  \right.
\end{equation}

It is easy to see that (\ref{a2}) admits a global bounded solution which is also increasing. We claim the following result: there exists $C_1>0, a_1>0$ such that
\begin{equation}
\label{f1n}
f(z)= L+C_1+ \al a_1 z + \mathcal{O} (\al e^{-\frac{ z}{\sqrt{2}}}),
\end{equation}
\begin{equation}
\label{fz1}
f_z (z)=  \al a_1  + \mathcal{O} (\al e^{-\frac{ |z|}{\sqrt{2}}}),
\end{equation}
\begin{equation}
\label{fzz3}
f_{zz} (z)=  \mathcal{O} (\al e^{-\frac{|z|}{\sqrt{2}}}).
\end{equation}
Since $f^{'} \geq 0$, it is easy to see that (\ref{f1n})-(\ref{fz1})
is a consequence of (\ref{fzz3}). We just need to establish
(\ref{fzz3}).  To this end, we note that for all $z\in \R$ we have $
\sqrt{L^2+ |z|^2} \geq \frac{1}{\sqrt{2}} L +
\frac{1}{\sqrt{2}}|z|.$ Because of our choice of $\alpha$ at
(\ref{alphachoice}), we have $e^{-\sqrt{L^2+ z^2}}\leq \al
e^{-\frac{1}{\sqrt{2}}|z|}.$ This implies that \be\label{z10} f_{zz}
= \mathcal{O}(\al e^{-\frac{1}{\sqrt{2}} |z|})\ee which proves
(\ref{fzz3}).

\subsection{Bounded solvability of (\ref{a2}) on  $\R$}
In this section we study the linearized operator of (\ref{a2s}), around a solution $f$ of (\ref{a2s}). Let $g$ be an even continuous, bounded function. Consider the following linear equation
\begin{equation}
  \label{lin} \mathcal{Q} (\psi):=\psi''-\frac{\partial \Psi_L}{\partial f} \psi= g ~~~~\text{in } \R\\
\end{equation}
We analyze the solvability of the linear problem in $\psi\in C^{2, \mu}_{\theta}(\R)$, given $g\in C^{0, \mu}_{\theta}(\R).$

Note that asymptotically we have
\begin{equation}
-\frac{\partial \Psi_L}{\partial f} \sim \frac{f}{\sqrt{f^2+z^2}} e^{-\sqrt{f^2+z^2}}
\end{equation}
\begin{remark}\label{lin3}
For the homogeneous equation,  there are two fundamental solutions $\psi_1$ and $\psi_2$ satisfying
\begin{equation}
  \label{lin1}
\left\{\begin{aligned}
     \psi_1'' -\frac{\partial \Psi_L}{\partial f} \psi_1&= 0 &&\text{in } \R\\
    \psi_1(0) &= 0 &&\text{ }\\
     \psi_1'(0) &= 1,
    \end{aligned}
  \right.
\end{equation}
\begin{equation}
  \label{lin2}
\left\{\begin{aligned}
     \psi_2''- \frac{\partial \Psi_L}{\partial f} \psi_2&= 0 &&\text{in } \R\\
    \psi_2(0) &= 1 &&\text{ }\\
     \psi_2'(0) &= 0.
    \end{aligned}
  \right.
\end{equation}

Note that $\psi_1$ is odd while $\psi_2$ is  even.  We now claim that $ \psi_1^{'} (+\infty) \not = 0$. In fact, suppose $ \psi_1^{'} (+\infty)=0$. Since $ f_z$ satisfies
\begin{equation}
  \label{lin1-1}
\left\{\begin{aligned}
     f_z''-\frac{\partial \Psi_L}{\partial f} f_z&= \frac{\partial \Psi_L}{\partial z} <0 &&\text{in } \R\\
    f_z(0) &= 0 &&\text{ }\\
     f_{zz} (0) &= \Psi_L (0, 0),
    \end{aligned}
  \right.
\end{equation}
 and $f_z>0$, we see that by the Maximum Principle $\psi_1 >0$. Then if $\psi_1^{'} (+\infty)=0$, then we have $ \int_0^{+\infty} \frac{\partial \Psi_L}{\partial z}  \psi=0$ which is impossible.

Thus $\psi_1 $ grows like $ c z$ as $+\infty$. This implies that $\psi_2$ must be a constant at $+\infty$.

We define the one dimensional space called the {\em deficiency
subspace} $\mathcal{E}=\{\chi \psi_{1} \}$  and $\chi$ is a  smooth
cut off function such that
\begin{equation}
  \label{lin0}
 \chi(z)=\left\{\begin{aligned}
     1 &&\text{if } z>1 \\
    0 &&\text{if } z<0.
    \end{aligned}
  \right.
\end{equation}

Moreover,  we define the norm on $C^{2, \mu}_{a}(\R)\oplus \mathcal{E}$ to be such that \be \no \|(\psi, c \chi \psi_1) \|_{C^{2, \mu}_{a}(\R)\oplus \mathcal{E}}= \|\psi\|_{C^{2, \mu}_{a}(\R)}+ |c|.\ee
\end{remark}

\begin{lemma}\label{2.2} [Linear Decomposition Lemma]
Let $f$ be the unique solution of (\ref{a2s}). The mapping
$$\mathcal{Q}: C^{2, \mu}_{\theta}(\R)\oplus \mathcal{E} \rightarrow C^{0, \mu}_{\theta}(\R) $$
$$\psi\mapsto \psi''-\frac{\partial \Psi_L}{\partial f} \psi$$
is an isomorphism.
\end{lemma}
\begin{proof}
Let $\|g\|_{C^{0, \mu}_{\theta}(\R)}< +\infty.$ Then it is easy to see that by the method of variation of constants the following function
\begin{equation}
\psi = \mathcal{R} (g)=\psi_1 (z) \int_0^z \psi_2 g + \psi_2 (z) \int_z^{+\infty} \psi_1 g
\end{equation}
is a solution to
\begin{equation}
\psi\mapsto \psi''-\frac{\partial \Psi_L}{\partial f} \psi=g, \psi^{'} (0)=0
\end{equation}

We claim that $\psi= \mathcal{R}(g)\in C^{2, \mu}_{\theta}(\R)\oplus \mathcal{E}.$
In fact, we simply write
\begin{eqnarray}
 \mathcal{R} (g) &=& \mathcal{R}_1 (g) + \mathcal{R}_2 (g) \chi \psi_1 \nonumber \\
  &= & \psi_1 (z) (1-\chi) \int_0^z \psi_2 g -\chi \psi_1 \int_z^{+\infty} \psi_2 g + \psi_2 (z) \int_z^{+\infty} \psi_1 g - \psi_1 (z) \no\\
 & & +\int_0^{+\infty} \psi_2 g \chi \psi_1 (z)
\end{eqnarray}
where $ \mathcal{R}_2 (g)= \int_0^{+\infty} \psi_2 g$.

Clearly we have
\begin{equation}
\| \mathcal{R}_1 (g) \|_{C^{2,\mu}_\theta (\R)} \leq C \| g \|_{C^{0, \mu}_\theta (\R)}, \  \ | \mathcal{R}_2 (g) | \leq C \| g \|_{C^{0, \mu}_\theta (\R)}
\end{equation}

\end{proof}

\begin{remark} Moreover, the space $\mathcal{E}$ can also be described as a {\em parameter space} for the linear problem $\mathcal{Q},$ since the elements are potentially occurring parameters for the {\em Jacobi field} that is those elements $\psi$ such that $\mathcal{Q}(\psi)=0.$
\end{remark}

\subsection{Solvability of another differential equation}
In an analogous way we look for even solutions of
\be\label{reso} e''+ \lambda_1 e= k(z)\ee
where $k$ is even with $\|k (\cosh z)^{\theta}\|_{C^{0, \mu}(\R)}< +\infty.$ We are interested in solution which decays to zero at $+\infty.$
Since (\ref{reso}) is a resonance problem, we impose the following orthogonality condition \be\label{resor}\int_{0}^{\infty} k(z) \cos(\sqrt{\lambda_1} z) dz =0\ee to prove existence and  uniqueness of solutions. Using the method of variation of parameters the solution of (\ref{reso}) can be written as $\mathcal{S}(k)=e$ where
\bea \label{6.34} \mathcal{S}(k)&=& \frac{1}{\sqrt{\lambda_1}}\sin(\sqrt{\lambda_1} z)\int_{z}^{\infty} k(t)\cos(\sqrt{\lambda_1} t) dt\no\\&-&\frac{1}{\sqrt{\lambda_1}}\cos(\sqrt{\lambda_1} z) \int_{z}^{\infty} k(t)\sin(\sqrt{\lambda_1} t) dt\eea
Furthermore, we have
\be \label{6.35} \|e (\cosh z)^{\theta}\|_{C^{2, \mu}(\R)} \leq C< +\infty.\ee
\subsection{Location of the spikes}
Let $\boldsymbol{\xi}= (\xi_{1}, \xi_{2} \cdots, \cdots)$ be a sequence of points satisfying
\be\label{spi}\xi_{2}= 2\xi_{1}+\mathcal{O}(1)\ee and for all $j\geq 2$ \be\label{spi1}\xi_{j+1}- \xi_{j} = \xi_{j}- \xi_{j-1}.\ee
Then we obtain for all $j\geq 1$
\be\label{spi2} \xi_{j}= j L+ \mathcal{O}(1).\ee

\subsection{Invertibility of the operator associated with the Toeplitz matrix}
Let $\boldsymbol{\xi}= (\xi_{i})_{i\geq 1}.$
We define an operator $T: \R^{\infty}\rightarrow \R^{\infty}$ such that
$T=(T(\xi_{i}))_{j}$ where
\begin{equation}
  \label{todm}
(T(\xi_{i}))_{j}=\left\{\begin{aligned}
        2 \xi_{j} &&\text{if }& j=i\\
         -\xi_{j} &&\text{if }& j= i\pm 1\\
     0 &&\text{ otherwise.}
         \end{aligned}
  \right.
\end{equation}
Our main goal is given $\boldsymbol{\chi}= (\chi_{1}, \cdots, \chi_{j}\cdots )$ we want to solve $T(\boldsymbol{\xi}) =  \boldsymbol{\chi}.$
Using the fact that (\ref{a1}) we define a weighted norm $\boldsymbol{\xi}=(\xi_{i})_{i=1}^{\infty}$ by
$$\|\boldsymbol{\xi}\|_{\al}=\|(\xi_1, \xi_2, \xi_3, \cdots )\|_{\al} = \max_{i}\al^{-i} |\xi_{i}|.$$
Let $$\Om= \{ \boldsymbol{\xi}=(\xi_{1}, \xi_{2}, \cdots \xi_{i}, \cdots): \|\boldsymbol{\xi}\|_{\al}<+\infty \}.$$
\begin{lemma} The operator $T$ has an inverse in $\Om$, whose norm is $\mathcal{O}(\al).$
\end{lemma}
\begin{proof} For any  $\|\boldsymbol{\chi}\|_{\al}<+\infty$, we define
$$\xi_{j}= \sum_{k=j}^{\infty} (k-j) \chi_{k}.$$ Let $I$ denote the operator defined by the above expression. Then $I$ is an operator inverse of $T.$ Clearly we have
\be |\xi_{j}|\leq \|\boldsymbol{\chi}\|_{\al} \sum_{k=j}^{\infty} (k-j) \al^{k+1}\leq C \al^{(j+1)} \|\boldsymbol{\chi}\|_{\al} \ee
This implies \be \no \al^{-j}|\xi_{j}|\leq C \al \|\boldsymbol{\chi}\|_{\al} \ee
\be \no  \|\boldsymbol{\xi}\|_{\al}\leq  C \al\|\boldsymbol{\chi}\|_{\al}.\ee Note that $C$ is independent of $\al.$
\end{proof}

\subsection{Idea of the construction} We are actually looking for bump line solution of (\ref{2}) whose asymptotic behavior is determined by the curve
$$\gamma= \{(x, z): x= f(z)\},$$
which asymptotically behaves as straight lines having negative exponential growth in the second order. Then it turns out that $f$ satisfies a second order differential equation, given by (\ref{a2}). Moreover, by (\ref{1.3}) we have  $f(z)= \beta+ \al a_1 |z|+  \mathcal{O}_{C^{\infty}}(\al (\cosh z)^{-\frac{1}{\sqrt{2}}})$
for some $\beta>0.$ Define $\theta=\frac{1}{\sqrt{2}}.$ Also note that the solution of (\ref{a2}) is unique and since $f(z), f(-z)$ are
solutions to (\ref{a2}) we must have $f(z)= f(-z)$ for all $z\in \R.$
Let $Z$ be a positive eigenfunction of
\be\label{h2} \varphi_{xx}+ (p\om^{p-1}-1) \varphi=\lambda_{1}\varphi\ee
corresponding to the principal eigenvalue $\lambda_1$ where explicitly $$Z(x)=\frac{\om^{\frac{p+1}{2}}(x)}{\int_{\R} \om^{p+1 }dx};~~~~~~~\lambda_{1}=\frac{1}{4}(p+3) (p-1)$$ and in particular, the asymptotic behavior of $\om$ and $Z$ at infinity are given by
$$\om(x)\sim e^{-|x|}+ \mathcal{O}_{C^{\infty}(\R)} (e^{-2|x|}) $$ and
$$Z(x)\sim e^{-\frac{p+1}{2}|x|}+ \mathcal{O}_{C^{\infty}(\R)} (e^{-(p+1)|x|}).$$
Consider the {\em Dancer's solution} of (\ref{2}) as $\om_{\de} (x,z)$
$$\om_{\de} (x,z)= \om(x)+ \de Z(x) \cos (\sqrt{\lambda_1} z) + \mathcal{O}(\de^2) e^{-|x|}$$ where $ |\de|$ is sufficiently small.

\subsection{Modified Fermi coordinates near the bump line}
Let $f$ be a solution of (\ref{a2}). We choose $v\in \mathcal{E}$ such that
\be \label{2.16} v=c \chi \psi_1, |c| \leq \al^{1+k_{1}}\ee
where $k_{1}$ is a small number to be chosen later. Now we define a model bump curve  as
$$\overline{\gamma}= \{\textbf{x}= (x, z)\in \R^2: x= \overline{f} (z)= f(z)+ v(z)\}$$ where $f$ is the solution of (\ref{a2}).
Then we define the local coordinate as a vector tuple $(T, N)$ where unit tangent $$T=\frac{1}{\sqrt{1+ ( \overline{f}')^2}} ( \overline{f}', 1)$$
and the unit normal to the curve $$N=\frac{1}{\sqrt{1+ ( \overline{f}')^2}} (1, -  \overline{f}').$$ Let $\text{z}$  be the arc length defined as $$\text{z}= \int_{0}^{z}\sqrt{1+ (\overline{f}'( s))^2 }ds $$ which is an increasing function of $z$
and let $q(\text{z})$ be the corresponding arc length parameter. Note that $q(\text{z})\in \R^2.$
It turns out that the asymptotic behavior of the bump line at infinity is not exactly linear but has an exponentially small correction. This correction needs to be determined and in fact this is the key step in the paper which involves the linearized operator discussed in Remark \ref{lin3}. To describe this small perturbation we consider a fixed function $h$  \be \label{2.17} \|h\|_{C^{2, \mu}_{\theta}(\R)}\leq \al^{1+k_{2}}\ee for some $k_{2}>0$ small.

A neighborhood of the curve $\overline{\gamma} $ can be parametrised in the following way
\be \label{w14}{\textbf{x}}= X (\text{x}, \text{z})= q (\text{z})+ (\text{x}+ h( \text{z}))N (\text{z})\ee
where $t= \text{x}+ h( \text{z})$ is the {\em signed } distance to the curve $\overline{\gamma}.$
Define a set
$$V_{\zeta}=\{\textbf{x}=(x, z): |x|\leq \zeta \sqrt{1+ z^2}\}$$
for small $\zeta.$ In fact the Fermi coordinates of the curve is  defined as long as the map $(t, \text{z}) \mapsto {\textbf{x}}$ is one-one. The  asymptotic behavior of the curvature of  $\overline{\gamma}_{}$ as $|\text{z}|\rightarrow +\infty$ is given by
$$\kappa (\text{z})\sim \al^{}(\cosh \text{z})^{ -\theta}.$$ Furthermore, we can show that for $\zeta$ and $\al$ sufficiently small the Fermi coordinates are well defined around $\overline{\gamma}(\text{z})$ as long as
\be |t|\leq \zeta \sqrt{1+ \text{z}^2}.\ee
Also we have
\be \textbf{x}\in V_{\zeta} ~~\Rightarrow |\text{x}|= |t- h( \text{z})|\leq \zeta \sqrt{1+ \text{z}^2}\ee
where $\textbf{x}= X(\text{x}, \text{z}).$
Moreover, we define
\be X^{\star} f(\text{x}, \text{z})= f ~~\text{o}~~X (\text{x}, \text{z}).\ee
Furthermore,  we have
\be\label{imes} x= \text{x}(1+ \mathcal{O}(\al^2))+ \text{z}\mathcal{O}(\al)\ee and
\be \label{imes1} z= (1+ \mathcal{O}(\al^{2}))\text{z}.\ee
\subsection{Laplacian in the shifted coordinates}
The curvature $\kappa $ of the curve $\bar{\gamma} $ which is given by
\be \kappa =\frac{ \bar{f}''(z)}{(1+  (\bar{f}'( z)))^{\frac{3}{2}}}.\ee
We define $A$ by
\be\no A:=1- (\text{x}+h)\kappa.\ee
Then the Laplacian in terms of the new coordinates reduces to
\be\label{lap} \De= \frac{1}{A}\bigg\{\pa_{\text{x}}\bigg(\frac{A^2+ (h')^2}{A}\pa_{\text{x}}\bigg)- \pa_{\text{z}}\bigg(\frac{ h'}{A}\pa_{\text{x}}\bigg)- \pa_{\text{x}}\bigg(\frac{ h'}{A}\pa_{\text{z}}\bigg)+ \pa_{\text{z}}\bigg(\frac{1}{A}\pa_{\text{z}}\bigg)\bigg\}.\ee
Then (\ref{lap}) can be written as
\be\label{lap1}  \De =\pa _{\text{x}}^2+  \pa _{\text{z}}^2+ a_{11}\pa _{\text{x}}^2+ a_{12}\pa _{\text{x}} \pa _{\text{z}}+ a_{22}\pa_{\text{z}}^2 + b_1 \pa _{\text{x}}+ b_{2} \pa_{\text{z}}^2\ee
where
\be\no a_{11}=\frac{ (h')^2}{A^2}, a_{12}=\frac{2  h'}{A^2}, a_{22}=\frac{1-A^2}{A^2} \ee
\be \label{b11} b_{1}= \frac{1}{A^3}(-\kappa A^2- h''A+  (h')^2 \kappa-  (\text{x}+h)h' \kappa)\ee and
\be \label{b2-1} b_{2}= \frac{1}{A^3}((h+ \text{x})\kappa).\ee
Note that here we have
$$\kappa= \mathcal{O}_{C^{2, \mu}_{ \theta}(\R)}(\al), \kappa'= \mathcal{O}_{C^{2, \mu}_{  \theta}(\R)}(\al^2)$$
and consequently we have
\begin{equation}
  \label{2.43}
\left\{\begin{aligned} a_{11}=  \mathcal{O}_{C^{0, \mu}_{  \theta}(\R)}(\al^2) , a_{12}=  \mathcal{O}_{C^{0, \mu}_{\theta}(\R)}(\al),  a_{22}=  \mathcal{O}_{C^{0, \mu}_{ \theta}(\R)}(\al(1+ |\text{x}|)) \\
     b_{1}= \mathcal{O}_{C^{0, \mu}_{\theta}(\R)}(\al(1+|\text{x}|)), b_{2}=  \mathcal{O}_{C^{0, \mu}_{\theta}(\R)}(\al(1+|\text{x}|)).
    \end{aligned}
  \right.
\end{equation}
\subsection{Approximate solution} In this section we develop the approximate solution. Firstly we take a {Dancer} solution and the homoclinic solution. These two solutions need to be glued together by some cut-off function. In this way the amplitude and the phase shifts of the ends do not change but instead remain fixed. To achieve an extra degree of freedom a function whose local form is given by $e(\text{z})Z(\text{x})$ is added to our approximation. \\
Precisely, we consider $e\in C^{2, \mu}_{\theta} (\R)$ such that \be \label{2.18} \|e\|_{C^{2, \mu}_{\theta}(\R)}\leq C \al^{2+ k_{3}}\ee
where $k_3$ will be chosen later. In addition, we will use a real parameter $\de_{}$  such that
\be\label{2.19} |\de|\leq  \al^{1+ k_{4}}.\ee
We define the following notations
\bea X_{}^{\star} \om_{\de}(\text{x}, \text{z})&=& \om_{\de_{}}(\text{x}, \text{z})\no\\ X_{}^{\star} \om(\text{x}, \text{z})&=& \om(\text{x}) \\ X_{}^{\star} Z (\text{x}, \text{z})&=& Z(\text{x})\no \eea
where $\om$ is the homoclinic solution, $\om_{\de}$ and $Z$ being the Dancer solution and the principle eigenfunction of (\ref{h2}) respectively.
Now we choose $\Xi_{}$ and $\Xi_{0}$ be nonnegative even cut-off function such that
$$\Xi_{}(t)+  \Xi_{0}(t) =1; ~~~~~~~~~\forall ~~~~~t \in \R$$ with
$$supp~ \Xi_{}=(-\infty, -1) \cup (1, +\infty), supp~\Xi_{0}=(-2, 2).$$

Also let $$X^{\star}\Xi_{}(\text{x}, \text{z})= \Xi_{}( \text{z}), X_{}^{\star}\Xi_{0}(\text{x}, \text{z})= \Xi_{0}(\text{z}).$$
Now we introduce
$\text{w}= \Xi_{}\om_{}+ \Xi_{0}\om .$ Let $\boldsymbol{\chi}$ be such that $$\|\boldsymbol{\chi}\|_{\al}\leq C \al^{k_{5}}$$ where $k_{5}$ is a small positive number. Define
$$\om_{j}(x,z)= \om_{0}(x-\xi_{j}-\chi_{j},z) \text{ and } \om_{j, x}(x,z)= \om_{0,x}(x-\xi_{j}-\chi_{j},z).$$
We thereby define the approximate solution  of (\ref{1}) in $V_{\zeta}$ as
\be \label{ap0} \bar{w}(\hat{\text{x}})=\text{w}+ e( \text{z})Z.\ee
 Now we intend to define a global approximation. Let $\eta_{\zeta}$ be a smooth cutoff function such that $supp~\eta_{\zeta}\subset V_{\zeta}$ such that $\eta\equiv 1$ in $V_{\frac{\zeta}{2}}$ and $\boldsymbol{\xi}$ satisfying (\ref{spi2}), then we define the global approximation as
\bea \label{ap} \textbf{w}&=& \eta_{\zeta}(\text{w}+ e(\text{z})Z)+ \sum_{j=1}^{\infty}\om_{j}(x,z) \no \\&=& \eta_{\zeta}\bar{w}+ \sum_{j=1}^{\infty}\om_{j}(x,z).\eea
Notice that ${\bf w}$ depends on $f, v, h, \de, \boldsymbol{\chi} .$

\subsection{The key estimates} In this section we precisely derive some key estimates concerning the interaction of spikes and the interaction of the front with the spike.
First note that $\om_{0}$  is radial and the asymptotic behavior of $\om_{0}$ at infinity is given by
$$\lim_{r\rightarrow \infty}e^{r} r^{\frac{1}{2}} \om_{0}(r)= A_{0}>0;~~~~  \text{and } \lim_{r\rightarrow \infty} \frac{\om_{0}'(r)}{\om_{0}(r)}=-1.$$
We have the following key estimates:
Let $\hat{x}= (x, z).$  Let $\vec{e}_{1}=(0,1),$ then  we have
\begin{eqnarray}\label{ke1} \int_{\R^2} \om_{0}^{p}(\hat{x})\om_{0}'(|\hat{x}+ L \vec{e}_1|) \frac{{x+L }}{|\hat{x}+ L\vec{e}_1|}d\hat{x}&=&- A_{0}\int_{\R^2} |\hat{x}+L \vec{e}_{1}|^{-\frac{1}{2}} e^{-|\hat{x}+ L\vec{e}|}\frac{x+ L}{|\hat{x}+ L\vec{e}_{1}|} \om_{0}^{p}(\hat{x})\no\\&=&- A_{0}\int_{\R^2} |\hat{x}+L \vec{e}_{1}|^{-\frac{1}{2}} e^{-|\hat{x}+ L\vec{e}_{1}|}\frac{x+ L}{|\hat{x}+ L\vec{e}_{1}|} \om_{0}^{p}(\hat{x})\no\\&=&- A_{0}\int_{\R^2} L^{-\frac{1}{2}}e^{-L} e^{L-|x+ L\vec{e}_{1}|}\frac{x+ L}{|\hat{x}+ L\vec{e}_{1}|} \om_{0}^{p}(\hat{x})\no\\&=&
- A_{0}L^{-\frac{1}{2}} e^{-L}\int_{\R^2} \om_{0}^{p} (\hat{x})e^{-L} e^{L-|\hat{x}+ L\vec{e}_{1}|}\frac{x+ L}{|\hat{x}+ L\vec{e}_{1}|}\no\\&=& -\om_{0}({L})(1+ O(L^{-1})) \int_{\R^2} e^{- |x|}\om_{0}^{p}(\hat{x}) d\hat{x}\no\\&=& -\gamma_{0}\om_{0}(L)\bigg(1 + \mathcal{O}\bigg(\log \frac{1}{\al}\bigg)^{-1}\bigg)\end{eqnarray}
where $\gamma_{0}= \int_{\R^2} e^{- |x|}\om_{0}^{p}(\hat{x}) d\hat{x}.$  In this case we consider $L$ to be either $\xi_{j+1}- \xi_{j}$ or $\xi_{j}- \xi_{j-1}$ where $j\geq 2.$ Similarly we can show there exists $\gamma_1>0$ such that
\be \label{ke3} \int_{\R^2} \om^{p}(x)\om_{0}'(|\hat{x}+ L \vec{e}_1|) \frac{{x+L }}{|\hat{x}+ L\vec{e}_1|}d\hat{x}=-\gamma_1 e^{-|\xi_{1}- f(0)|}\bigg(1 + \mathcal{O}\bigg(\log \frac{1}{\al}\bigg)^{-1}\bigg).\ee

\section{Proof of Theorem \ref{1.1}}
Let $\eta$ and $\eta_{j}$ be smooth cut-off function such that \begin{equation}
  \label{rho1}
\eta(s)=\left\{\begin{aligned}
        1 &&\text{if }& |s|\leq \frac{3}{4} \log\frac{1}{\al}\\
        0 &&\text{if }& |s|> \frac{7}{8 } \log \frac{1}{\al}\\
         \end{aligned}
  \right.
\end{equation} and
\begin{equation}
  \label{rho}
\eta_{j}(s, t)=\left\{\begin{aligned}
        1 &&\text{if }& |(s, t)-\xi_{j}\vec{e}_{1}| \leq \frac{3}{4} \log \frac{1}{\al}\\
        0 &&\text{if }& |(s, t)-\xi_{j}\vec{e}_{1}|  > \frac{7}{8} \log \frac{1}{\al}.
         \end{aligned}
  \right.
\end{equation}
Define $X^{\star}_{} \eta= \eta({\text{x}})$ and $X^{\star} \om'=\om'(\text{x}).$
We are looking for solutions of (\ref{1}) of the form $u= \textbf{w}+\varphi$ where $\varphi$ is a {\em small} perturbation of $\textbf{w}.$ Substituting the value of $u$ in (\ref{1}), we obtain
\be\label{g1} \De (\textbf{w}+\varphi)- (\textbf{w}+\varphi)+ (\textbf{w}+\varphi)^{p}=0\ee
where $\textbf{w}=\textbf{w}(\al, v, h, e, \de, \boldsymbol{\chi} )$ for some $\varphi\in C^{2, \mu}_{\sigma, \theta }(\R^2)\oplus C^{1}_{\sigma}(\R^2)$ and $v\in \mathcal{E}.$
We can formally write (\ref{g1}) as $$\mathcal{L}(\varphi)+ S(\textbf{w}) + N(\varphi)=0 \text{ in } \R^2$$
where
$$\mathcal{L}:=\De -1+ p \textbf{w}^{p-1}$$ and
$$N(\varphi):= (\textbf{w}+\varphi)^p- \textbf{w}^p- p\textbf{w}^{p-1} \varphi$$ with
$$S(\textbf{w}):= \De \textbf{w}- \textbf{w}+ \textbf{w}^{p}.$$
Hence we should write (\ref{g1}) as a fixed point problem for the nonlinear function
$$\varphi+ \mathcal{L}^{-1} (S(\textbf{w}) + N(\varphi))=0$$
provided $\mathcal{L}^{-1}$ is a suitable bounded operator. But $\mathcal{L}$ will have in general an unbounded inverse as $L\rightarrow +\infty.$
Also note that near the bump line the operator
$L_{0}= \pa_{x}^2+ \pa_{z}^2- p\om^{p-1}+ 1$ which has a bounded kernel spanned by $\om', Z(x)\cos\sqrt{\lambda_{1}}z$ and  $Z(x)\sin\sqrt{\lambda_{1}}z$ and near a spike the kernel of $L_{1}= \De-1+ p\om_{0}^{p-1}$ is spanned by $\om_{0, x}.$
To get rid of this difficulty we consider a {\em nonlinear projected problem}
\be\label{nlp}  \mathcal{L}(\varphi)= S(\textbf{w})+ N(\varphi)+ \sum_{j=1}^{\infty} c_{j} \eta_{j}\om_{j, x}+ d(z)\eta \om'+ m(z)\eta Z.\ee
In the following sections we will describe:\\
(1) How to solve (\ref{nlp}) for unknown $\varphi, $ $c= (c_{1}, c_{2}, \cdots \cdots)$, $d, m$ with the given parameters $v, h, e, \de$ and $\boldsymbol{\chi}.$\\
(2) Secondly we have to choose the parameters in such a way that $c,d, m$ are zero.

\section{Linear Theory}
The local structure of $\mathcal{M}_{3}(\R^2)$ near the curve $\gamma$ and the spikes are closely related to the study of 
whether $\mathcal{L}$ is actually injective or not. If $\mathcal{L}$ is not injective, we need to determine its kernel. 
We first study two simplified linear operators
\be \label{z20} L_0(\varphi)= \varphi_{zz}+\varphi_{xx} -\varphi + p\om^{p-1}\varphi\ee and
\be \label{z21} L_{1} (\varphi)=\De \varphi -\varphi + p\om_{0}^{p-1}\varphi\ee
where $\om$ is the unique  solution of (\ref{a1}) and decays exponentially; and $\om_{0}$ is the unique positive
solution  of (\ref{2}).
Note that $\om',$ $Z(x) \cos\sqrt{\lambda_1} z$ and
$Z(x) \sin\sqrt{\lambda_1} z$ are solutions to $L_0(\varphi)=0.$  In Lemma \ref{nd}, we prove that
indeed the converse also hold.

\begin{lemma}\label{nd} Let $\varphi$ be a bounded solution of \begin{equation} L_0(\varphi)=0.\no \end{equation} Then $\varphi \in span\{ \om'(x), Z(x)\cos\sqrt{\lambda_1} z, Z(x)\sin\sqrt{\lambda_1} z\}.$
\end{lemma}
\begin{proof}  This follows from Lemma 7.1 of \cite{DKPW1}.
\end{proof}

\begin{lemma}\label{nd1} Let $\varphi$ be a bounded solution of  \begin{equation} L_1(\varphi)=0.\no \end{equation}
satisfying $\varphi(x, z)= \varphi(x, -z).$ Then $\varphi= c\om_{0,x}$ for some $c\in \R.$
\end{lemma}
\begin{proof} Since the kernel of $L_1$ consists of  $\om_{0, x}$ and $\om_{0, z},$ see \cite{NW}, the result follows trivially from the fact that $\varphi$ is even in $z-$ variable.
\end{proof}
By Lemma \ref{nd} and \ref{nd1} we define the orthogonality conditions as \be\label{orth}\int_{\R} \varphi(x,z) \om'(x) dx = 0= \int_{\R} \varphi(x,z) Z(x) dx~~~\forall z\in \R\ee and \be\label{orth1}\int_{\R^2} \varphi(x,z) \om_{0, x}(x, z) dx dz = 0.\ee
\begin{lemma} \label{L} Let $\varphi$ be a bounded solution of \begin{equation}\label{s1} L_0(\varphi)=k\end{equation} satisfying (\ref{orth}). Then $\|\varphi\|_{\infty}\leq C \|k\|_{\infty}$ for some $C>0.$
\end{lemma}
\begin{proof} This follows from Lemma 7.2 of \cite{DKPW1}.
\end{proof}

\begin{remark} Note that Lemma \ref{L} implies that $\|\varphi\|_{L^{\infty}(\R^2)}\leq C \|(\cosh z)^{\sigma} k\|_{L^{\infty}(\R^2)}.$
\end{remark}
\begin{lemma} \label{a0} Assume that $\sigma\in (0, 1)$ be fixed. Then there exist $C>0$ such that for any solution of $L_{0}(\varphi)= k$ satisfies
\be \|(\cosh x)^{\sigma} \varphi\|_{C^{2, \mu}(\R^2)}\leq C \|(\cosh x)^{\sigma} k\|_{C^{0, \mu}(\R^2)}.\ee
\end{lemma}
\begin{proof} This is again Lemma 7.3 of \cite{DKPW1}. \end{proof}
\begin{lemma}\label{c2} Assume that $\sigma\in (0,1)$. Then there exists $a_{0}>0$ such that for all $a\in (0, a_{0}]$
 there exists a constant $C_{a}>0$ but remains bounded as $a$ tends to zero, such that
\be\no  \|(\cosh x)^{\sigma} (\cosh z)^{a} \varphi\|_{L^{\infty}(\R^2)}\leq C_{a} \|(\cosh x)^{\sigma}  (\cosh z)^{a} k\|_{L^{\infty}(\R^2)}.\ee
\end{lemma}
\begin{proof} This follows from Lemma 7.4 of \cite{DKPW1}.
\end{proof}

\subsection{Surjectivity} As far as the existence of  solution of (\ref{s1}) and (\ref{orth}) is concerned we assume that
\be\label{cu1} \int_{\R}k (x, z) \om_{x}(x) dx= 0 \ee \be \label{cu2} \int_{\R}k(x, z)Z(x) dx=0\ee
for all $z\in \R,$ we prove the following proposition.

\begin{prop} Assume that $\sigma\in (0, 1)$  be fixed. Then there exists $a_{0}>0$ such that for all $a\in (0, a_{0}];$
 there exists a constant $C_{a}>0$  such that for all $k$ satisfying the orthogonality conditions (\ref{cu1}), (\ref{cu2})  and  $$L(\varphi)= k,$$ with
\be \| (\cosh x)^{\sigma} (\cosh z)^{a} k\|_{C^{0, \mu}(\R^2)}<+\infty \ee
implies
\be  \|(\cosh x)^{\sigma} (\cosh z)^{a} \varphi\|_{C^{2, \mu}(\R^2)}\leq C_{a}  \| (\cosh x)^{\sigma} (\cosh z)^{a} k\|_{C^{0, \mu}(\R^2)}.\ee
\end{prop}
\begin{proof} The main idea is to prove the result for functions which are $R$ periodic in the $z-$ variable. We consider the problem
$$L_{0}(\varphi)=k$$ with the orthogonality conditions (\ref{orth}) and (\ref{orth1}).  We will apply an approximation argument.
Let $\varphi(x, z)$ be a $\xi $ periodic function in the $z$ variable where $\xi>0.$ Define $\R_{\xi}^2= \R\times \frac{\R}{\xi \mathbb{Z}}.$
Then we have
$$\int_{\R^2_{\xi}}[ |\nabla \varphi|^2-(p \om^{p-1}-1) \varphi^2]\geq \frac{\lambda_1}{2}\int_{\R^2_{\xi}} \varphi^2$$
Hence given $k\in L^{2}(\R^2_{\xi})$ satisfying
\be \no  \int_{\R^2_{\xi}}k  \om_{x} =0 =\int_{\R^2_{\xi}} k Z(x) \ee
by Lax-Milgram lemma there exists a unique $\varphi\in H^{1}(\R^2_{\xi})$ such that
$$\|\varphi\|_{H^{1}(\R^2_{\xi})} \leq C \|k\|_{L^{2}(\R^2_{\xi})}.$$
Moreover, by elliptic regularity, we have
$$\|\varphi\|_{L^{\infty}(\R^2_{\xi})} \leq C (\|k\|_{L^{2}(\R^2_{\xi})}+ \|k\|_{L^{\infty}(\R^2_{\xi})}).$$
Suppose in addition $k$ satisfies (\ref{cu1}) and (\ref{cu2}) we obtain
$$\int_{0}^{\xi} \bigg(\int_{\R} \varphi \om_{x} dx\bigg) \psi_{zz} dz=0$$
and
$$\int_{0}^{\xi} \bigg(\int_{\R} \varphi Z(x) dx\bigg) \psi_{zz} dz=0. $$
Hence
$$z\mapsto \int_{\R} \varphi \om_{x}dx  ~~~ \text{and }  z\mapsto \int_{\R} \varphi Z dx$$
do not depend on $z$ since its integral over $[0,\xi]$ is $0,$ we conclude that $\varphi$ satisfies (\ref{s1}) and (\ref{orth}).

Hence we can apply Lemma \ref{nd1} and \ref{c2} to obtain the estimate
$$\|(\cosh x)^{\sigma} \varphi\|_{L^{\infty}(\R^2_{\xi})}\leq C \|(\cosh x)^{\sigma} k\|_{L^{\infty}(\R^2_{\xi})}.$$
where $C>0$ is a constant independent of $\xi.$ Given any $k$ satisfying the condition of the Proposition. Let
$$k_{\xi}= k\chi_{\R^2_{\xi}}$$
where $\chi$ denotes the characteristic function. Let $\varphi_{\xi}$ be the corresponding solution to
$$L_{0}\varphi_{\xi}= k_{\xi}$$
Elliptic estimates with compactness arguments yield we can pass through the limit as $\xi\rightarrow +\infty$, there exists a
bounded solution $\varphi$ of $ L_{0} \varphi= k.$
\end{proof}

\section{Linear Theory for Multiple interfaces}
\subsection{Gluing Procedure} In this section we decompose the nonlinear projected problem (\ref{nlp}) into four coupled equations. We define
\begin{equation} \label{rho3}
\rho(s)=\left\{\begin{aligned}
        1 &&\text{if }& |s|\leq \frac{7}{8} \log \frac{1}{\al}\\
        0 &&\text{if }& |s|> \frac{15}{16} \log \frac{1}{\al}\\
         \end{aligned}
  \right.
\end{equation} and
\begin{equation}
  \label{rho4}
\rho_{j}(s, t)=\left\{\begin{aligned}
        1 &&\text{if }& |(s, t)-\xi_{j}\vec{e}_{1}| \leq \frac{7}{8} \log \frac{1}{\al}\\
        0 &&\text{if }& |(s, t)-\xi_{j}\vec{e}_{1}| >\frac{15}{16} \log \frac{1}{\al} .\\
         \end{aligned}
  \right.
\end{equation}
Moreover, we define $X_{}^{\star}\rho= \rho({\text{x}}).$ Using the definition, we obtain $\rho_{j}\eta_{j}=\rho_{j}$
and $\rho_{j}\rho_{k}=0$ for $j\neq k.$ Similarly we have $\rho \eta =\rho.$ Moreover, $\rho \eta_{j}=0$ for every $j\in \N.$
Note that we are looking for solutions of (\ref{nlp}) of the form
\be \label{eax} \varphi=\sum_{j=1}^{\infty} \eta_{j} \phi_j + \eta \phi+ \psi  \ee
where $\psi=\psi_1+ \psi_2.$ Then for $j\in \N$, we have
\bea\label{b1} \rho_{j}[\mathcal{L}\phi_{j} - \frac{1}{2}(S(\text{w})+ N)- c_{j} \om_{j,x}]+ (\mathcal{L}- \De+ 1)\psi_{1} \rho_{j} =0 \eea
\be\label{b2} \rho[\mathcal{L}\phi- \frac{1}{2}(S(\text{w})+ N)-d(z) \om_{}'- m(z) Z]+ [\mathcal{L}- \De+1] \psi_{2} \rho=0\ee
$\psi_1$ and $\psi_2$ satisfy the following equation,
\bea\label{b3a} (\De-1) \psi_{1}&=& \frac{1}{2}(1- \sum_{j=1}^{\infty} \eta_{j})(S(\text{w})+ N) \no\\&-& \sum_{j=1}^{\infty}(\mathcal{L}(\phi_{j}\eta_{j})- \eta_{j}\mathcal{L}(\phi_{j}))- (1- \sum_{j=1}^{\infty}\eta_{j})(\mathcal{L}- \De +1)\psi_{1} \eea and
\bea\label{b3b} (\De-1) \psi_{2} &=& \frac{1}{2}(1- \eta)(S(\text{w})+ N)\no\\&-& (\mathcal{L}(\phi\eta) - \eta \mathcal{L}(\phi))-(1- \eta)(\mathcal{L}- \De +1)\psi_{2}\eea
where $N= N(\sum_{j=1}^{\infty} \eta_{j} \phi_j + \eta \phi+ \psi ).$ This is a coupled system and the coupling terms are of the higher order in $\al$.
Note that (\ref{b2}) can be written as
\be\label{sw1} [\pa ^2_{\text{x}}+ \pa ^2_{\text{z}}- F'(\om)]X_{}^{\star}\phi_{}= X_{}^{\star}k + X_{}^{\star} ( d \rho \om ')+ X_{}^{\star} (m \rho  Z)\ee
where
\bea\label{sw2}  X_{}^{\star}k&=& X_{}^{\star}[\frac{\rho}{2}(S(\textbf{w})+N)]-  X^{\star}[\rho  (\mathcal{L}- \De+1) \psi_{2}]\no\\&-& X_{}^{\star}\rho  (\mathcal{L}(\phi))+ X_{}^{\star}\rho [\pa ^2_{\text{x}}+ \pa ^2_{\text{z}}- F'(\om)]X_{}^{\star}\phi \eea
Define ${\Phi}= (\phi_{1}, \phi_{2}, \cdots \cdots).$ Let the right hand side of  the equation (\ref{b3a}) and (\ref{b3b}) be $Q_{1}= Q_{1}(\Phi, \psi_{1})$ and $Q_{2}= Q_{2}(\phi, \psi_{2})$ respectively. Then equation (\ref{b3a}) and (\ref{b3b}) reduces to
\be\label{back1} (\De -1)\psi_{1}= Q_{1}\ee
\be\label{back2} (\De -1)\psi_{2}= Q_{2}\ee
We will call (\ref{back1}) and (\ref{back2}) as the {\em background system}. We will first solve the background system. Then for the given solution $(\psi_{1}, \psi_{2})$, we solve the initial equations (\ref{b1}) and (\ref{b2}).

\subsection{Error of the initial approximation} For $0<\mu\leq 1,$ we define
the weighted norms
$$\|\varphi\|_{C^{l, \mu}_{\sigma, \theta}(\R^2)}=\sup_{\hat{x}\in \R^2} \bigg((\cosh x)^{\sigma}(\cosh z)^{\theta} \|\varphi\|_{C^{2, \mu}(B_{1}(\hat{x}))}\bigg).$$
We also define the norms
$$\|\varphi\|_{\sigma}= \sup_{(x,z)\in \R^2}\bigg(\sum_{i=1}^{\infty}
e^{-\sigma |(x, z)-\xi_i\vec{e}_{1}| }\bigg)^{-1}|\varphi (x, z)|$$ and
\be \no \boldsymbol{\chi}= (\chi_{1}, \cdots, \chi_{k}, \cdots).\ee
\be \no \|\boldsymbol{\chi}\|= \max_{i} \al^{-i}|\chi_{i}|\ee

\begin{prop}\label{comp0} For $i=1, 2;$ $S({\bf{w}}^{(i)})= S({\bf{w}}, v, h^{(i)}, e^{(i)}, \de_{}, \boldsymbol{\chi}^{(i)})$ is a continuous function of $v,\de_{}$  and satisfies
\be\label{ee} \|X_{}^{\star}(\rho S({\bf{w}})) \|_{C^{0, \mu}_{\sigma, \theta}(\R^2)} \leq C\al.\ee
Moreover, it is a Lipschitz function of $h$, $e$ and $\boldsymbol{\chi}$;
\bea \|(X_{}^{(1)})^{\star}\rho^{(1)}S({\bf{w}}^{(1)})&-&(X_{}^{(2)})^{\star} \rho^{(2)}S({\bf{w}}^{(2)})\|_{C^{0, \mu}_{\sigma, \theta}(\R^2)}\leq C (\| h^{(1)}- h^{(2)}\|_{C^{2, \mu}_{\theta}(\R)}\no \\&+& \| e^{(1)}- e^{(2)}\|_{C^{2, \mu}_{\theta}(\R)}+ \al \| \boldsymbol{\chi}^{(1)}-  \boldsymbol{\chi}^{(2)}\|_{\al}).\eea
\end{prop}
So far we have estimated the error near the bump line. The other two propositions deal with the estimate of the norm in the complement of the set $supp\quad\rho$ and the estimation of the error near the spikes. Note that in $\R^2\setminus V_{\zeta}$ we have $S(\textbf{w})= S(\sum_{j=1}^{\infty}\om_{j}).$ Let us denote
\be V_{\zeta}^{\perp}= V_{\zeta}\setminus \text{supp}~\eta.\ee
\begin{prop}\label{comp1} Then we have in $V_{\zeta}^{\perp}$ \be\label{ee1-4} \| S({\bf{w}}) \|_{C^{0, \mu}_{\theta}(V_{\zeta}^{\perp})} \leq C\al^{1+ \frac{3}{4}\sigma} .\ee
Moreover,
\bea\label{ee2} &&\|(S({\bf{w}}^{(1)})- S({\bf{w}}^{{(2)}})\|_{C^{0, \mu}_{\sigma, \theta}(V_{\zeta}^{\perp})}\leq C \al^{\frac{3}{4}\sigma} (\| h^{(1)}- h^{(2)}\|_{C^{2, \mu}_{\theta}(V_{\zeta}^{\perp})}\no \\&+& \| e^{(1)}- e^{(2)}\|_{C^{2, \mu}_{\theta}(V_{\zeta}^{\perp})}+ \al\| \boldsymbol{\chi}^{(1)}-  \boldsymbol{\chi}^{(2)}\|_{\al}).\eea
\end{prop}
\begin{prop}\label{comp2} In $\R^2\setminus V_{\zeta}$ we have \be\label{ee1} \| S({\bf{w}}) \|_{\sigma} \leq C\al^{}.\ee
Moreover,
\bea\label{ee2-2} &&\|(S({\bf{w}}^{(1)})- S({\bf{w}}^{(2)})\|_{\sigma}\leq C \al  \|\boldsymbol{\chi}^{(1)}-  \boldsymbol{\chi}^{(2)}\|_{\al}.\eea
\end{prop}
\begin{proof}[Proof of Propositions \ref{comp0}, \ref{comp1}, \ref{comp2} ] We write
\be S(\textbf{w})= \rho S(\textbf{w})+ \sum_{j=1}^{\infty}\rho_{j} S(\textbf{w})\ee
Let $U_{1}:= V_{\frac{\zeta}{2}}\cap \{\text{x}+ \text{z}\geq 0\}.$
Then using the approximation we have \be \textbf{w}=\text{w} + e(\text{z})Z(\text{x})+ \sum_{j=1}^{\infty}\om_{j} \ee and using the fact  $\De \om_{j}+ F(\om_{j})=0.$  As a result, we have
\bea \no S(\textbf{w})&=&\De \textbf{w}+ F(\textbf{w})\no\\&=& \underbrace{\De \text{w}+ F(\text{w})}+ \underbrace{(\De + F'(\text{w}))e(\text{z})Z}
   \no\\&+& \underbrace{\{F(\textbf{w})-\sum_{j=1}^{\infty} F(\om_{j})- F(\text{w})- F'(\text{w})e Z\}}\no\\&=& E_{1}+ E_2+ E_{3}.\eea
Using Taylor's expansion we obtain
\bea \text{w}^{p}&=& \Xi  \om_{\de}^{p}+ \Xi_{0} \om^{p}+ (\om_{\de}+ \Xi(\om_{\de}-\om))^{p}- \Xi(\om_{\de}+ (\om_{\de}-\om_{}))^{p}- \Xi_{0} \om^{p}
\no\\&=& \Xi  \om_{\de}^{p}+ \Xi_{0} \om^{p}+ \mathcal{O}_{C^{0, \mu}(U_{1})}(\de_{}^2 )(\cosh \text{x})^{-2}(\cosh \text{z})^{-\theta}.\eea
Also note that
\be \pa_{\text{z}}\Xi_{}(\text{z})= \Xi_{}'(\text{z}), ~~~ \pa_{\text{z}}^2\Xi_{}(\text{z})=\Xi_{}''(\text{z}) \ee
with $|\de |\leq \al^{1+k_{4}}$ and since  $\om$ is not a function of $\text{z}$ we obtain $ \pa_{\text{z}} \om=0.$
Moreover, if we denote the operator $S=\De- \pa^2_{\text{x}}- \pa^2_{\text{z}}$ then we have
\bea E_{1}&=& S(\Xi \om_{\de}+ \Xi_{0}\om)+ 2 [\pa_{\text{x}} \Xi \pa_{\text{x}}\om_{\de}+ \pa_{\text{z}} \Xi_{0} \pa_{\text{z}}\om_{}]\no\\&+& 2 [\pa_{\text{x}}^2 \Xi\om_{\de}+ \pa_{\text{z}}^2 \Xi_{0} \om_{}]+ \mathcal{O}_{C^{\infty}(\R^2)}(|\de|^2) (\cosh \text{x})^{-2} (\cosh \text{z})^{-\theta}\eea
Note that the first term in the above expression is of the order $\al$ due to the fact of (\ref{2.43}).
Hence we have
\be \no \| E_1\|_{C^{0, \mu}_{\sigma, \theta}(U_{1})} \leq C\al\ee
Moreover,\be \no \|E_{1}\|_{C^{0, \mu}_{\theta}(V_{\zeta}^{\perp})}\leq C \al^{1+ \frac{3}{4}\sigma} \ee
and this follows due to the fact that $V_{\zeta}^{\perp}= V_{\zeta}\setminus V_{\frac{\zeta}{2}}$ we have  $|\text{x}|\geq \frac{3}{4}\log\frac{1}{\al}.$
The estimate for $E_{2}$ follows similarly. \\ Now we estimate $E_{3}.$
For $(\text{x}, \text{z})\in V_{\zeta}$ we have $\text{w}\gg e(\text{z}) Z + \om_{j}$ and hence
\be F( \textbf{w})= F(\text{w}) + F'(\text{w})(\textbf{w}- {\text{w}})+O(\text{w}^{2}(\textbf{w}-\text{w})^{p-2}) \ee
Hence we have
\bea E_{3} &=&  \{F'(\text{w})(\textbf{w}- \text{w})-\sum_{j=1}^{\infty}F(\om_{j})\} + O(\text{w}^{p-2}(\textbf{w}-\text{w})^{2})\no\\&=& p \text{w}^{p-1} \bigg(\sum_{j=1}^{\infty}\om_j \bigg) - \sum_{j=1}^{\infty}\om_j^{p}+ O(\text{w}^{p-2}(\text{w}-\text{w})^{2})\no.\eea
When $0< \sigma< (p-1)$ and  by (\ref{w5})  and the fact that $\text{x}\sim (x-f_{}(z))$ and $\text{z}\sim z$ we have,
\bea{(p-1) |x-f(z)|}+ |x-\xi_i| &=& \sigma |x-f(z)|+ (p-1-\sigma) |x-f(z)|+  |x-\xi_i| \no \\&\geq & \sigma |x-f(z)|+ \min\{ (p-1-\sigma), 1 \} \{|x- f_{}(z)|+ |x-\xi_i|\}\no\\&\geq & \sigma |x-f_{}(z)|+ \min\{ (p-1-\sigma), 1 \}|(f_{}(z)-\xi_i, z)|\no \\&= & \sigma |x-f_{}(z)|+ \min\{ (p-1-\sigma), 1 \}\sqrt{(f_{}(z)-iL )^2+ z^2}\no\\&\geq &  \sigma |x-f_{}(z)|+ \min\{ (p-1-\sigma), 1 \} \sqrt{L ^2+ z^2}\no\\&\geq & \sigma |x-f(z)|+  \min\{p-1-\sigma, 1\}\frac{L}{\sqrt{2}}+ \theta|z|\eea
Also note that from (\ref{w5}) we have,
\bea\label{wu} (p-\sigma)|(x, z)-\xi_j\vec{e}_{1}|&=& (p-\sigma)\sqrt{(x-jL)^2+ z^2}\geq (p-\sigma)\sqrt{\frac{L^2}{4}+ z^2} \no\\&\geq & \frac{(p-\sigma)L}{\sqrt{2}} + \theta |z|.\eea
Further note that
\bea \no|\text{w}^{p-2}(\textbf{w}-\text{w})^{2}|&=& \text{w}^{p-2} \bigg(e(\text{z})Z+ \sum_{j=1}^{\infty} \om_{j}\bigg)^{2}\no\\&\leq & C  (\cosh \text{x})^{-(p-2)}(\al^{4+ 2k_{3}}(\cosh \text{z})^{-2\theta} (\cosh \text{x})^{-(p+1)} + e^{-2 |(x, z)-\xi_{j}\vec{e}_{1}|})\no\\&\leq & C\al^{} (\cosh \text{x})^{-\sigma} (\cosh \text{z})^{-\theta} \eea
This implies that
\be\no  \|\text{w}^{p-2}(\textbf{w}-\text{w})^{2}\|_{C^{0, \mu}_{\sigma, \theta}(U_{1})}\leq C \al^{}\ee
Hence \be\no  \|E_{3}\|_{C^{0, \mu}_{\sigma, \theta}(U_{1})}\leq C \al^{}.\ee
Similarly we have \be \|E_{3}\|_{C^{0, \mu}_{\sigma, \theta}(V_{\zeta}^{\perp})}\leq C \al^{1+\frac{3}{4}\sigma}.\ee
Now define $$\mathcal{A}_{j}= \bigg\{(x, z)\in \R^2: |(x-\xi_{j}, z)|\leq \frac {L}{2}\bigg\}$$ where $j\geq 1.$
Then we have  in $\R^2\setminus V_{\zeta}$
\be S(\textbf{w})= \sum_{j\geq 1}S(\textbf{w})\chi_{\mathcal{A}_{j}}\ee
and if we expand near the spike $(\xi_{i}, 0)$ we have using mean value theorem \bea S(\textbf{w})&=& S\bigg(\sum_{j=1}^{\infty}\om_{j}\bigg)\no\\&=& F\bigg(\sum_{j=1}^{\infty}\om_{j}\bigg)- \sum_{j=1}^{\infty} F(\om_{j})\no\\&=& \bigg(\sum_{j=1}^{\infty}\om_{j}\bigg)^{p}- \sum_{j=1}^{\infty} \om_{j}^{p}\no \\&\sim& p \sum_{i\neq j}\om_{i}^{p-1}\om_{j}\no \\&\sim& p \sum_{i\neq j}e^{-(p-1-\sigma)\sqrt{(x-\xi_{i})^2+ z^2}}e^{-\sigma |(x, z)-\xi_{j}\vec{e}_{1}|} e^{-{|\xi_{j}-\xi_{i}|}}\no \\&\sim& p \sum_{i\neq j}e^{-(p-1-\sigma)\sqrt{(x-\xi_{i})^2+ z^2}} e^{-{|(j-i)|L}} e^{-\sigma |(x, z)-\xi_{j}\vec{e}_{1}|}.\eea
This implies
\be \no |S(\textbf{w})|\leq C e^{-L} \sum_{j=1}^{\infty}e^{-\sigma |(x, z)-\xi_{j}\vec{e}_{1}|}\ee and hence we have
\be\no \|S(\textbf{w})\|_{\sigma}\leq C  e^{-L}= C \al.\ee
\end{proof}

\subsection{Existence of solution for the background system}
 In order to solve (\ref{back1}) and (\ref{back2}) we will use the Banach fixed point theorem. Moreover, we assume that
\be \label{ke4}\sum_{j=1}^{\infty}\| e^{\sigma |(x, z)- \xi_{j}\vec{e}_{1}|}\phi_{j}\|_{L^{\infty}(\R^2)}<+\infty \ee and
\be\label{ke5} \|X ^{\star}\phi\|_{C^{2, \mu}_{\sigma, \theta}(\R^2)}< +\infty.\ee

\begin{lemma}\label{lback1} Assume that (\ref{ke4}) holds. Then there exists a unique solution of (\ref{back1}) such that
\be \label{5.17}\|\psi_{1}\|_{\sigma}+ \|\nabla \psi_{1}\|_{\sigma }\leq C \al^{\frac{3 }{4} \sigma} (\al+ \sum_{j=1}^{\infty} (\|\phi_{j}\|_{\sigma, j}+ \|\nabla \phi_{j}\|_{\sigma, j})).\ee
In addition $\psi_{1}$ is a continuous function of the parameter $v, h, e, \de$ and $\boldsymbol{\chi}$ and a Lipschitz function of $\phi_{j}$ and also of the parameters $e, h$ and $\boldsymbol{\chi}$ satisfies the following estimates
\bea && \|\psi_{1}(\Phi^{(1)})- \psi_{1}(\Phi^{(2)})\|_{\sigma}+ \|\nabla \psi_{1}(\Phi^{(1)})- \nabla \psi_{1}(\Phi^{(2)}))\|_{\sigma}\no\\&\leq& C \al^{\frac{3 }{4} \sigma} (\|\Phi^{(1)}- \Phi^{(2)}\|_{\sigma}+ \|\nabla \Phi^{(1)}- \nabla \Phi^{(2)}\|_{\sigma}) \eea
\bea \|\psi_{1}(h^{(1)}, e^{(1)}, \boldsymbol{\chi}^{(1)})- \psi_{1}(h^{(2)}, e^{(2)}, \boldsymbol{\chi}^{(2)}) \|_{\sigma}&\leq& C \al^{\frac{3 }{4} \sigma} (\|h^{(1)}- h^{(2)}\|_{C^{2, \mu}_{\theta}(\R)}\no \\&+& \|e^{(1)}- e^{(2)}\|_{C^{2, \mu}_{\al \theta}(\R)}+ \al\|\boldsymbol{\chi}^{(1)}- \boldsymbol{\chi}^{(2)}\|_{\al}).\eea
\end{lemma}
\begin{proof}
Define $\|\Phi\|_{\sigma}=\sum_{j=1}^{\infty}\| e^{\sigma |(x, z)-
\xi_{j}\vec{e}_{1}|}\phi_{j}\|_{L^{\infty}(\R^2)}$ and $$\|\nabla
\Phi\|_{\sigma}=\sum_{j=1}^{\infty}\| e^{\sigma |(x, z)-
\xi_{j}\vec{e}_{1}|}\nabla \phi_{j}\|_{L^{\infty}(\R^2)}.$$  We have
\be\no (\De -1)\psi_{1}= Q_{1}.\ee For the time being we assume that
$\|Q_{1}\|_{\sigma}< +\infty$ then
$$|Q_1|\leq  C \sum_{j=1}^{\infty} e^{-\sigma |(x, z)-\xi_j \vec{e}_{1}|}.$$
Using barrier and elliptic estimates we obtain
$$|\psi_{1}(x, z)|+ |\nabla \psi_{1}(x, z)|\leq C \sum_{j=1}^{\infty} e^{-\mu |(x, z)-\xi_j \vec{e}_{1}|}.$$
This implies that
\be\no \|\psi_{1}\|_{\sigma}+ \|\nabla\psi_{1}\|_{\sigma}\leq C \|Q_{1}\|_{\sigma}.\ee
Next we estimate the size of $Q_{1}$ and also its dependence on $\Phi= (\phi_{1}, \phi_{2}, \cdots)$ and $h, e, \boldsymbol{\chi}.$ We assume that
$$\|\Phi\|_{\sigma}+ \|\nabla \Phi\|_{\sigma}<+\infty.$$
We now estimate $Q_1$. Then we have
\bea \label{b10} |Q_{1}|&\leq& C \bigg(\al^{1+ \frac{3 }{4} \sigma} + \al^{\frac{3 }{4} \sigma} \sum_{j=1}^{\infty}(\|{\phi}\|_{\sigma, j}+  \|\nabla {\phi}\|_{\sigma, j})+ \al^{\frac{3}{4} \sigma} \|\psi_{1}\|_{\sigma}\bigg)\no \\&\times & \sum_{j=1}^{\infty} e^{-\sigma|(x, z)-\xi_j \vec{e}_{1}|}\eea
This implies
\be  \|Q_{1}\|_{\sigma}\leq C \al^{\frac{3 }{4} \sigma} \bigg(\al^{}+  \sum_{j=1}^{\infty}(\|{\phi}\|_{\sigma, j}+  \|\nabla {\phi}\|_{\sigma, j}) \bigg)+ \al^{\frac{3 }{4} \sigma} \|\psi_{1}\|_{\sigma}.\ee
Hence given $\Phi$, using a standard fixed point theorem there exists $\psi_{1}=\psi_{1}(\Phi)$ satisfying (\ref{back1}). Moreover,
\be \no \|\psi_{1}(\Phi)\|_{\sigma} \leq C \al^{\frac{3 }{4} \sigma} \bigg(\al+  \sum_{j=1}^{\infty}(\|{\phi}\|_{\sigma, j}+  \|\nabla {\phi}\|_{\sigma, j}) \bigg).\ee
Since $Q_{1}(\Phi, .)$ is a uniform contraction in the second variable and it is continuous we conclude that $\psi_{1}$ is also a continuous function and we conclude that $\psi_{1}$ is continuous function of $v, h, e, \de$ and $\boldsymbol{\chi}$. Moreover, it easily follows
\bea \no \| \psi_{1}(\Phi^{(1)})- \psi_{1}(\Phi^{(2)})\|_{\sigma}+ \|\nabla \psi_{1}(\Phi^{(1)})- \nabla \psi_{1}(\Phi^{(2)})\|_{\sigma}&\leq& C \al^{\frac{3 }{4} \sigma} (\|\Phi^{(1)}- \Phi^{(2)}\|_{\sigma}\\\no&+& \|\nabla \Phi^{(1)}- \nabla \Phi^{(2)}\|_{\sigma} ) \eea
\end{proof}

\begin{lemma} \label{lback2} Assume that (\ref{ke5}) holds. Then there exists a unique solution of (\ref{back2}) such that
\be\label{5.42} \|(\cosh \mathrm{z})^{\theta}\psi_{2}\|_{C^{2, \mu}_{}(\R^2)}\leq C \al^{\frac{3 }{4} \sigma}(\al+ \|X_{}^{\star}\phi\|_{C^{2, \mu}_{\sigma, \theta}(\R^2)}).\ee
In addition $\psi_{2}$ is a continuous function of the parameter $v, h, e, \de$ and $\boldsymbol{\chi}$ and a Lipschitz function of $\phi$ and also of the parameters $e, h$ and $\boldsymbol{\chi}$ and satisfy the following estimates
\be \|(\psi_{2}(\phi^{(1)})- \psi_{2}(\phi^{(2)})) (\cosh \mathrm{z})^{ \theta}\|_{C^{2, \mu}(\R^2)}\leq C \al^{\frac{3 }{4} \sigma} \|X^{\star}_{}(\phi^{(1)}- \phi^{(2)})\|_{C^{2, \mu}_{\sigma,\theta}(\R^2)}\ee
\bea \|(\psi_{2}(h^{(1)}, e^{(1)} )- \psi_{2}(h^{(2)}, e^{(2)})) (\cosh \mathrm{z})^{\theta}\|_{C^{2, \mu}(\R^2)}&\leq& C \al^{\frac{3 }{4} \sigma}(\|h^{(1)}- h^{(2)}\|_{C^{2, \mu}_{\theta}(\R)}\no \\&+& \|e^{(1)}- e^{(2)}\|_{C^{2, \mu}_{\theta}(\R)}+ \al\|\boldsymbol{\chi}^{(1)}- \boldsymbol{\chi}^{(2)}\|_{\al})\eea
\end{lemma}
\begin{proof} We have $$(\De-1)\psi_{2}= Q_{2}.$$ For the time being consider \be \|(\cosh \mathrm{z})^{\theta} Q_{2}\|_{C^{0, \mu}(\R^2)}< +\infty.\ee
Then by regularity theory we have
\be\no  \|\psi_{2}\|_{C^{2, \mu}(\R^2)}\leq C \|Q_{2}\|_{C^{0, \mu}(\R^2)}\ee
We are required to prove that
\be\label{req}  \|\psi_{2}(\cosh \mathrm{z})^{\theta} \|_{C^{2, \mu}(\R^2)}\leq C \|Q_{2}(\cosh {\text{z}})^{\theta}\|_{C^{0, \mu}(\R^2)}\ee
In order to so we define a barrier of the form
\be \no \psi_{\nu}= (\cosh \mathrm{z}(z))^{-\theta}+ \nu \bigg[\cosh\frac{x}{2}+ \cosh\frac{z}{2}\bigg] \ee
where $\nu\geq 0$ is sufficiently small. In fact we have
\be (\De-1)\psi_{\nu}\leq -\frac{1}{4}\psi_{\nu}\ee
and hence $\psi_{ \nu}$ is a super solution of  $\De-1.$ Moreover define
\be \vartheta_{ \nu, M}= M \| Q_{2}(\cosh \mathrm{z})^{\theta}\|_{C^{0, \mu}(\R^2)} \psi_{\nu}+\psi_{2} \ee
where $M>0$ is large such that
\bea (\De-1) \vartheta_{\nu, M}&\leq& -\frac{M}{4}\| Q_{2}(\cosh \mathrm{z})^{\theta}\|_{C^{0, \mu}(\R^2)} \psi_{\nu}+ Q_{2}\no
\\&\leq & -\frac{M}{4}\| Q_{2}(\cosh \mathrm{z})^{\theta}\|_{C^{0, \mu}(\R^2)} \psi_{\nu}+ \|Q_{2}(\cosh \mathrm{z})^{\theta}\|_{C^{0, \mu}(\R^2)} (\cosh \mathrm{z})^{-\theta}\\\no&\leq& 0\eea
Letting $\nu\rightarrow 0$ we obtain
\be\no \psi_{2}(\cosh \mathrm{z})^{\theta} \leq C \|Q_{2}(\cosh \mathrm{z})^{\theta}\|_{C^{0, \mu}(\R^2)} \ee
The lower estimate for (\ref{req}) can be obtained in a similar way.
Now we estimate $Q_{2}.$ Note that in $supp~~Q_{2}$ we have
\be\no |\mathrm{x}|\geq \frac{3}{4}\log \frac{1}{\al}.\ee
Note that we have already estimated the error $S({\bf w})$ and hence
\be \| (\cosh \mathrm{z})^{ \theta} S({\bf w})\|_{C^{0, \mu}(\R^2)} \leq C \al^{1+\frac{3 }{4} \sigma}.\ee
Moreover, using the fact that support of $\nabla \rho$ we have
\be \| (\cosh \mathrm{z})^{\theta} (\mathcal{L}(\eta\phi)- \eta \mathcal{L}(\phi ))\|_{C^{0, \mu}(\R^2)}\leq C \al^{\frac{3 }{4} \sigma}\|X_{}^{\star}\phi\|_{C^{0, \mu}_{\sigma, \theta}(\R^2)}\ee
and
\be p {\bf w}^{p-1}|\psi_{2}|\leq C \al^{\frac{3 }{4} \sigma} \|(\cosh \mathrm{z})^{\theta}\psi_{2}\|_{C^{0, \mu}(\R^2)} (\cosh \mathrm{z})^{-\theta} \ee
which finally yields
\be \no \| (\cosh\mathrm{z})^{\theta} Q_{2}(\phi, \psi_{2})\|_{C^{0, \mu}(\R^2)} \leq C \al^{\frac{3 }{4} \sigma}\{\al^{}+ \|X^{\star}_{}\phi\|_{C^{0, \mu}_{\sigma, \theta}(\R^2)}\}+  C \al^{\frac{3 }{4} \sigma}\|(\cosh \mathrm{z})^{\theta}\psi_{2}\|_{C^{0, \mu}(\R^2)}\ee
Hence given $\phi$ using a standard fixed point theorem there exists $\psi_{2}=\psi_{2}(\phi)$ satisfying (\ref{back2}). Moreover,
\be \no \| (\cosh \mathrm{z})^{\theta} \psi_{2}(\phi)\|_{C^{0, \mu}(\R^2)} \leq C \al^{\frac{3 }{4} \sigma} (\al^{}+ \|X^{\star}_{}\phi\|_{C^{0, \mu}_{\sigma, \theta} (\R^2)}).\ee
Since $Q_{2}(\phi, .)$ is a uniform contraction in the second variable and it is continuous and we conclude that $\psi_{2}$ is continuous function of $v, h, e, \de $ and $\boldsymbol{\chi}$. Moreover, it easily follows
\be \no \| (\cosh \mathrm{z})^{\theta} (\psi_{2}(\phi^{(1)})- \psi_{2}(\phi^{(2)}))\|_{C^{0, \mu}(\R^2)} \leq C \al^{\frac{3 }{4} \sigma} \|X^{\star}_{}\phi^{(1)}- X^{\star}_{}\phi^{(2)}\|_{C^{2, \mu}_{\sigma, \theta}(\R^2)}.\ee
\end{proof}

\subsection{Existence of solution for the initial problems}
For the time being consider $\hat{x}_{j}= (x, z)-\xi_{j} \vec{e}_{1}, ~~~j\in \N.$
\begin{lemma} Assume that $\|k_{j}\|_{\sigma, j}= \| e^{\sigma |\hat{x}_{j}|} k_{j}\|_{\infty}< +\infty.$ Then there exists $C>0$ independent of $j$ and $\sigma\in [0, 1)$ such that (\ref{ini1}) with (\ref{ini2}) satisfies
\be  \label{sw3}\|\phi_{j}\|_{\sigma, j}+ \|\nabla \phi_{j}\|_{\sigma, j}\leq C \|k_{j}\|_{\sigma, j}.\ee
Moreover, we have
\begin{equation} \label{2.4} |c_{j}|\leq C \|k_{j}\|_{\sigma, j}.\end{equation}
\end{lemma}
\begin{proof} We have $L(\varphi)= k_{j}$. Define $$\tilde{\phi}=\frac{\phi}{\|k_{j}\|_{\sigma}}.$$ Then we have $L(\tilde{\phi})= \tilde{k_{j}}$ where $\tilde{h}=\frac{k_{j}}{\|k_{j}\|_{\sigma}}.$ Note that it is enough to show that the estimate holds for sufficiently large $|\hat{x}|=|(x, z)|.$ Then there exists a $R>0$ such that for $|\hat{x}|\geq R$ we have $$p \om_{0}^{p-1}(|\hat{x}|)<\frac{1-\sigma^2}{2}.$$ Moreover, define
\be\no \overline{\phi}(\hat{x})=e^{-\sigma|\hat{x}|} \ee
Then $$L_{1}(\tilde{\phi}-M\overline{\phi})\geq 0$$ if $|x|>R$ and $|\tilde{\phi}|\leq M \overline{\phi}(R)= M e^{\sigma R}$ on $|\hat{x}|= R.$ Hence by the maximum principle, we obtain $|\tilde{\phi}|\leq M |\overline{\phi}|.$ Hence $|\phi|\leq M \|k_{j}\|_{\sigma}e^{-\sigma |\hat{x}|}.$ As a result we obtain
\be \no \|\phi\|_{\sigma}\leq M \|k_{j}\|_{\sigma}.\ee For the gradient estimate we define $\psi= e^{-\sigma|\hat{x}|} \phi.$ Then we have
\be L_{1}(\psi)+B(\psi)= k_{j} e^{\sigma |\hat{x}|}\ee where $B$ is an operator containing terms involving gradient and zero order terms, such that
$\|B\|_{\infty}$ is very small. Using local $C^{1}$ estimates we obtain
\be\no |\nabla \psi|\leq C \|k_{j}\|_{\sigma}.\ee
Hence for small $\sigma$ we obtain
$$ e^{-\sigma |\hat{x}|}|\nabla \phi|\leq C \|k_{j}\|_{\sigma}.$$ Hence the result.
Multiplying by $\om_{i, x},$ we have on integration by parts,
\begin{eqnarray*}0=\int_{\R^2} L_{1}(\varphi) \om_{i,x} dx dz&=& \int_{\R^2} k_{j}(x, z)\om_{i,x}+ c_i \int_{\R^2}\om_{i, x}^2 dx dz
\end{eqnarray*}
Hence we have
\begin{eqnarray*} |c_{j}|&\leq & \int_{\R^2} |k_{j}(x, z)||\om_{i,x}|\\
&\leq & C \|k_{j}\|_{\sigma}
\end{eqnarray*}
Hence the inequality follows easily.
\end{proof}

\begin{remark} For $p>2$, using the inequality
\be\label{ineq} ||a+b|^{p}- |b|^{p}- p |a|^{p-1} b|\leq C(p) \max\{|a|^{p-2} |b|^2, |b|^{p}\}.\ee
\end{remark}
\subsection{Existence of solution for \ref{b1}} As described earlier the derivation of solution of (\ref{b1}) is given by the linear theory of $L_{1}.$ \\
Note that we can write (\ref{b1}) as \be\label{ini1} (\De -1+
p\om_{j}^{p-1})\phi_{j}= k_{j}+ c_{j}\om_{j, x}\ee where
\be\label{hj} k_{j}= (S(\textbf{w})+N)\rho_{j}+
p\textbf{w}^{p-1}\psi_{1}\rho_{j}+ (p\textbf{w}^{p-1}-
p\om_{j}^{p-1})\phi_{j} \rho_{j}\ee with the condition of
orthogonality as \be \label{ini2}\int_{\R^2}\phi_{j}(x, z)\om_{j,
x}(x, z) dxdz= 0 \ee for all $j\in \N.$ Hence there exists a $C>0$
such that given $\|h\|_{\sigma}<+\infty$ and for some $\sigma\in
(0,1)$, there exists a unique bounded solution $\Phi=
\mathcal{T}_{}(h)$ to (\ref{ini1}) and (\ref{ini2}) which defines a
bounded linear operator of $h$ satisfying
\begin{equation}\no \|\Phi\|_{\sigma}+ \|\nabla \Phi\|_{\sigma}\leq C \|h\|_{\sigma}.\end{equation} This follows trivially by using the Fredholm alternative. Hence we write the linear operator $\mathcal{T}=(T_{1}, \cdots, \cdots )$ such that for each $j\in \N$ such that $\phi_{j}= T_{j}(h).$\\
Hence we can write
\be\label{cat13} \phi_{j} = T_{j}(\rho_{j}S(\text{w})+ \rho_{j}N(\phi)+ c_{j} \rho_{j}\om_{j,x} + (\mathcal{L}- \De+ 1)\psi_{1} \rho_{j}+ (p {\bf w}^{p-1}- p \om^{p-1}_{j})\phi_{j})\rho_{j}\ee
for some linear operator $T_{j}$; $j\in \N.$
Let $k=(k_{1}, k_{2}, \cdots).$
\begin{lemma}  Assume that \be\label{zs2} \sum_{j=1}^{\infty}(\|\phi_{j}\|_{\sigma}+ \|\nabla \phi_{j}\|_{\sigma})\leq \al^{\frac{3}{4}\sigma}.\ee
Then we have for all $j\in \N$
\be \|k_{j}\|_{\sigma}\leq C \al + C\al^{\frac{3}{4}\sigma}\bigg(\sum_{j=1}^{\infty}\|\phi_{j}\|_{\sigma}+ \|\nabla \phi_{j}\|_{\sigma}\bigg ).\ee
Moreover, the function $k_{j}$ is a Lipschitz function of $\Phi$ and satisfy
\be\label{lip1}  \|k_{j}(\Phi^{(1)})- k_{j}(\Phi^{(2)})\|_{\sigma}\leq C\al^{\frac{3}{4}\sigma} (\|\Phi^{(1)}- \Phi^{(2)} \|_{\sigma}+ \|\nabla \Phi^{(1)}- \nabla \Phi^{(2)} \|_{\sigma}).\ee Furthermore, we have
\be \|c\|_{\infty}\leq C \|k\|_{\sigma}.\ee
\end{lemma}
\begin{proof} From (\ref{hj}) we have
$$\|\rho_{j}S({\bf w})\|_{\sigma}\leq C \al.$$
Now \be\no \rho_{j} N(\varphi)= \rho_{j} N(\eta_{j}\phi_{j}+
\psi_{1})= N (\rho_{j} \phi_{j}+ \psi_{1}).\ee Hence from
(\ref{ineq}) \be |\rho_{j} N(\varphi)|\leq C  (|\phi_{j}|^2+
|\psi_{1}|^2).\ee As a result we have \be\no  \|\rho_{j}
N(\varphi)\|_{\sigma}\leq C (\|\phi_{j}\|_{\sigma}^2+
\|\psi_{1}\|_{\sigma}^2).\ee and hence \bea \no \|\rho_{j}
N(\varphi)\|_{\sigma}&\leq& C ((\|\Phi_{}\|_{\sigma}+ \|\nabla
\Phi_{}\|_{\sigma})^2+ \al^{\frac{3}{4}\sigma}(\al+
(\|\Phi_{}\|_{\sigma}+ \|\nabla \Phi_{}\|_{\sigma}))^2)\no \\&\leq &
C \al + C \al^{\frac{3}{4}}(\|\Phi_{}\|_{\sigma}+ \|\nabla
\Phi_{}\|_{\sigma}).\eea Then (\ref{5.17}) implies \bea\no
\|k_{j}\|_{\sigma}&\leq& C \al^{}+ C
\al^{1+\frac{3}{4}\sigma}\sum_{j=1}^{\infty}(\|\phi_{j}\|_{\sigma}+
\|\nabla\phi_{j}\|_{\sigma}))\no\\&+& \|\phi_{j}\|_{\sigma}
\bigg(\sum_{j\neq i}e^{- (p-2-\sigma)|\xi_{j}-\xi_{i}|}+ e^{-
(p-2-\sigma)|f(z)-\xi_{i}|}\bigg).\eea This implies that \bea
\|k_{j}\|_{\sigma}&\leq& C \al^{}+ C
\al^{\frac{3}{4}\sigma}(\|\Phi_{}\|_{\sigma}+ \|\nabla
\Phi_{}\|_{\sigma})+ \al^{p-2-\sigma}(\|\Phi_{}\|_{\sigma}+ \|\nabla
\Phi_{}\|_{\sigma}) \no
\\&\leq& C \al^{}+ C\al^{\frac{3}{4}\sigma} (\|\Phi_{}\|_{\sigma}+ \|\nabla
\Phi_{}\|_{\sigma})\eea provided we choose $\sigma$ is chosen small.
The Lipschitz dependence follows in a standard way.
\end{proof}

\begin{lemma}\label{lip3} The problem (\ref{cat13}) and (\ref{orth1}) has a unique solution $\phi$ such that
\be\label{con1}\|\Phi\|_{\sigma}+ \|\nabla\Phi\|_{\sigma} \leq C\al.\ee
Moreover,  the solution  is a continuous function of $v, h, e, \de$ and $ \boldsymbol{\chi}$ and a Lipschitz function of  $h$, $e$ and $\boldsymbol{\chi}.$ Furthermore, for every $j\in \N$, there exists $C>0$ independent of $j$ such that
\bea \|\phi_{j}(h^{(1)}, e^{(1)}, \boldsymbol{\chi}^{(1)})&-& \phi_{j}(h^{(2)}, e^{(2)}, \boldsymbol{\chi}^{(2)})\|_{\sigma}\no\\&\leq& C ( \|h^{(1)}- h^{(2)}\|_{C^{2, \mu}_{\theta}(\R)}+ \|e^{(1)}- e^{(2)}\|_{C^{2, \mu}_{\theta}(\R)}+ \al\|\boldsymbol{\chi}^{(1)}- \boldsymbol{\chi}^{(2)}\|_{\al}).\eea
\end{lemma}
\begin{proof} First note that from (\ref{zs2}) we have
\be \no \sum_{j=1}^{\infty}\|\phi_{j}\|_{\sigma}+ \|\nabla \phi_{j}\|_{\sigma}\leq C \al+ \al^{\frac{3}{4}\sigma}\bigg(\sum_{j=1}^{\infty}\|\phi_{j}\|_{\sigma}+ \|\nabla \phi_{j}\|_{\sigma}\bigg )\ee
which implies
\be \no \sum_{j=1}^{\infty}(\|\phi_{j}\|_{\sigma}+ \|\nabla \phi_{j}\|_{\sigma})\leq C \al \ee
 which implies that the operator $T_{j}$; $j\geq 1$ in (\ref{cat1}) is a uniform contraction in the set of functions satisfying (\ref{zs2}) as long as
(\ref{2.16}), (\ref{2.17}), (\ref{2.18}) and (\ref{2.19}) hold. In fact $\phi_{j}$ is a continuous function of $v, h, e, \de$ and $ \boldsymbol{\chi}$ and a Lipschitz  function of $h$, $e$ and $\boldsymbol{\chi}$ which follows from Lemma (\ref{lback1}) and Proposition \ref{comp0}. Hence by Banach fixed point theorem we obtain (\ref{con1}).
\end{proof}

\subsection{Existence of solution for \ref{ini1}} As described in the linear theory the derivation of solution of the (\ref{ini1}) is given in the linear theory of the operator in $L_{0}.$ This problem is basically reduced to a problem of fixed point
\be\label{cat1} X_{}^{\star}\phi = T(X^{\star}k + X^{\star} ( d\rho \om ')+ X^{\star} (m \rho  Z))\ee
where $d, m$ satisfy
\be\label{ortr1} d \int_{\R} X^{\star}(\om')^{2}\rho d \text{x}= -\int_{\R} X_{}^{\star}\om' k\rho d\text{x}\ee
\be\label{ortr2} m \int_{\R} X^{\star}Z^{2}\rho d\text{x}= -\int_{\R} X_{}^{\star}Z k\rho d\text{x}\ee
respectively.

\begin{lemma}\label{zs0}  Assume that \be\label{5.47} \|X^{\star}\phi\|_{C^{2, \mu}_{\sigma, \theta}(\R^2)}\leq \al^{\frac{3}{4}\sigma}.\ee
Then we have
\be \|X_{}^{\star}k\|_{C^{2, \mu}_{\sigma, \theta}(\R^2)}\leq C (\al + \al^{\frac{3}{8}\sigma} \|X_{}^{\star}\phi\|_{C^{2, \mu}_{\sigma, \theta}(\R^2)}).\ee
Moreover, the function $X^{\star}k$ is a Lipschitz function of $\phi$ and satisfy
\be\label{lip1-1}  \|X_{}^{\star}k(\phi^{(1)})- X_{}^{\star}k(\phi^{(2)})\|_{C^{2, \mu}_{\sigma, \theta}(\R^2)}\leq C \al^{\frac{3}{8}\sigma} \|X_{}^{\star}\phi^{(1)}- X_{}^{\star}\phi^{(2)} \|_{C^{2, \mu}_{\sigma, \theta}(\R^2)}.\ee Furthermore, we have
\be \|d\|_{C^{0,\mu}_{\theta}(\R)}+ \|m\|_{C^{0,\mu}_{\theta}(\R)}\leq C \|X_{}^{\star}k\|_{C^{0,\mu}_{\sigma, \theta}(\R^2)}.\ee
\end{lemma}
\begin{proof} The proof of the Lipschitz property (\ref{lip1-1}) is quite standard and left for an interested reader. We know that
\be \no  \|X^{\star}_{}(\rho S({\bf w}))\|_{C^{0, \mu}_{\sigma,  \theta}(\R^2)}\leq C\al.\ee
We need to estimate $X_{}^{\star}k$ given by (\ref{sw2}). We can rewrite (\ref{sw2}) as
\begin{eqnarray*}  X_{}^{\star}k&=& X_{}^{\star}\bigg[\rho S(\textbf{w})+\rho N\bigg(\sum\eta_{j}\phi_{j}+\eta \phi+ \psi\bigg)\bigg]-  X_{}^{\star}[\rho  (\mathcal{L}- \De+1) \psi_{2}]\no\\&-& X_{}^{\star}\rho  (\mathcal{L}(\phi))+ X_{}^{\star}\rho [\pa ^2_{\text{x}}+ \pa ^2_{\text{z}}- F'(\om)]X_{}^{\star}\phi \end{eqnarray*}
Using Lemma \ref{lback2} we obtain
\be\label{g2} \|X_{}^{\star}(\rho \psi_{2}) (\cosh \mathrm{z})^{ \theta}\|_{C^{2, \mu}(\R^2)}\leq C \al^{\frac{3 \sigma}{4} \sigma}(\al+ \|X_{}^{\star}\phi\|_{C^{2, \mu}_{\sigma, \theta}(\R^2)}).\ee
Note that using the definition of $\rho$ and $\psi$ we have
\be\no \rho N\bigg(\sum_{j=1}^{\infty}\eta_{j}\phi_{j}+\eta \phi+ \psi\bigg)= \rho N ( \eta \phi+ \psi_{2}).\ee
We obtain from (\ref{ineq})
\be |X_{}^{\star}\rho N|\leq C (|X_{}^{\star} \phi|^{2}+ |X_{}^{\star}(\rho \psi_{2})|^{2}).\ee
Note that $$supp~ (X_{}^{\star}\rho) \subset \bigg\{|\text{x}|\leq \frac{15}{16} \log \frac{1}{\al}\bigg\}.$$
We have from (\ref{g2})
\bea\label{g3} \|(\cosh \mathrm{x})^{\sigma}(\cosh \mathrm{z})^{\theta} X_{}^{\star}(\rho \psi_{2})\|^2_{C^{0, \mu}(\R^2)}&\leq& C \al^{-\frac{15}{8}\sigma}\|(\cosh \mathrm{z})^{\theta} X_{}^{\star}(\rho\psi_{2})\|^2_{C^{0, \mu}(\R^2)}\no \\&\leq & \al^{\frac{3\sigma}{2} -\frac{15\sigma}{8}}(\al+ \|X_{}^{\star}\phi\|_{C^{2, \mu}(\R^2)})^2\no\\&\leq & C\al^{-\frac{3\sigma}{8}}(\al+ \|X_{}^{\star}\phi\|_{C^{2, \mu}(\R^2)})^2\no \\&\leq & C\al^{-\frac{3\sigma}{8}}(\al+ \|X_{}^{\star}\phi\|_{C^{2, \mu}(\R^2)})^2.\eea
Hence from (\ref{g3}) we have
\be \| X_{}^{\star}(\rho N)\|_{{C^{0, \mu}_{\sigma, \theta}(\R^2)}}\leq C (\al^{2-\frac{3}{8}\sigma}+ \|X_{}^{\star}\phi\|^2_{{C^{2, \mu}_{\sigma,\theta}(\R^2)}}+ \al^{1-\frac{3}{8}\sigma}\|X_{}^{\star}\phi\|_{{C^{2, \mu}_{\sigma, \theta}(\R^2)}}).\ee
Next we estimate the term $X_{}^{\star}(\rho f'({\bf w})\psi_{2}).$ Note that $X^{\star}(\rho {\bf w}^{p-1})$  decays in the $\mathrm{x}$ variable like $(\cosh \mathrm{x})^{-(p-1)}$ we obtain
\bea \| X_{}^{\star} (\rho f'({\bf w}))\psi_{2}\|_{{C^{0, \mu}_{\sigma, \theta}(\R^2)}}&\leq& C \|(\cosh \mathrm{z})^{\theta} X^{\star}(\rho \psi_{2})\|_{C^{0, \mu}(\R^2)}\no\\&\leq & \al^{\frac{3\sigma}{4}}(\al+ \| X_{}^{\star}\phi\|_{C^{2, \mu}_{\sigma, \theta}(\R^2)}).\eea
In order to estimate the last terms we use (\ref{lap1}) to obtain
\be \| X_{}^{\star}[\rho(\De - \pa^{2}_{\mathrm{x}}-\pa^{2}_{\mathrm{z}})]\phi\|_{{C^{0, \mu}_{\sigma,  \theta}(\R^2)}}\leq C \al \| X_{}^{\star}\phi\|_{C^{2, \mu}_{\sigma, \theta}(\R^2)}\ee
and
\be \| X_{}^{\star}[\rho(f'({\bf w})- f'(\om))]\phi\|_{{C^{0, \mu}_{\sigma, \theta}(\R^2)}}\leq C \al \| X_{}^{\star}\phi\|_{C^{2, \mu}_{\sigma,  \theta}(\R^2)}.\ee
Orthogonality conditions (\ref{ortr1}) and (\ref{ortr2}) imply
\be \no \|d\|_{C^{0,\mu}_{\theta}(\R)}+ \|m\|_{C^{0,\mu}_{\theta}(\R)}\leq C \|X_{}^{\star}k\|_{C^{0,\mu}_{\sigma, \theta}(\R^2)}.\ee
The Lipschitz dependence follows in a standard way.
\end{proof}

\begin{lemma}\label{zs1} The problem (\ref{cat1}), (\ref{ortr1}) and (\ref{ortr2}) has a unique solution $\phi$ such that
\be\label{con1-1}\|X_{}^{\star}\phi\|_{C^{2, \mu}_{\sigma,\theta }(\R^2)}\leq C\al.\ee
\end{lemma}
\begin{proof} From (\ref{cat1}) we obtain by the fixed point theorem
\bea \|X_{}^{\star}\phi\|_{C^{2, \mu}_{\sigma,  \theta}(\R^2)}&\leq& C \|X^{\star}_{} k \|_{C^{0, \mu}_{\sigma,  \theta}(\R^2)}+ C \|X^{\star}_{}\rho   d \om' \|_{C^{0, \mu}_{\sigma, \theta}(\R^2)}\no \\&+& C\|X^{\star} \rho  m Z \|_{C^{0, \mu}_{\sigma, \theta}(\R^2)}\no\\&\leq& C  \|X^{\star}_{} k \|_{C^{0, \mu}_{\sigma, \theta}(\R^2)}.\eea
Using the Lemma \ref{zs0} we obtain,
\be \no \|X_{}^{\star}\phi\|_{C^{2, \mu}_{\sigma, \theta}(\R^2)}\leq C (\al+ \al^{\frac{3}{8}\sigma}\|X_{}^{\star}\phi\|_{C^{2, \mu}_{\sigma,  \theta}(\R^2)}).\ee
This implies
\be \no \|X_{}^{\star}\phi\|_{C^{2, \mu}_{\sigma, \theta}(\R^2)}\leq C \al.\ee
\end{proof}
\begin{lemma}\label{lip2} The solution of (\ref{cat1}), (\ref{ortr1}) and (\ref{ortr2})  is a continuous function of $v, h, e, \de $ and $\boldsymbol{\chi}$ and a Lipschitz function of  $h$, $e$ and $\boldsymbol{\chi}.$ Moreover, we have
\bea \|X_{}^{\star}\phi(h^{(1)}&,& e^{(1)}, \boldsymbol{\chi}^{(1)}, .)- X_{}^{\star}\phi(h^{(2)}, e^{(2)}, \boldsymbol{\chi}^{(2)},.)\|_{C^{2, \mu}_{\sigma, \theta}(\R^2)}\no\\&\leq& C \|h^{(1)}- h^{(2)}\|_{C^{2, \mu}_{\theta}(\R)}+ C \|e^{(1)}- e^{(2)}\|_{C^{2, \mu}_{\theta}(\R)}+ C \al\|\boldsymbol{\chi}^{(1)}- \boldsymbol{\chi}^{(2)}\|_{\al }\eea
\end{lemma}
\begin{proof} First note that from Lemma \ref{zs1} we have  $\|X_{}^{\star}\phi\|_{C^{2, \mu}_{\sigma, \theta }(\R^2)}\leq \al$ which implies that the operator $T$ in (\ref{cat1}) is a uniform contraction in the set of functions satisfying (\ref{5.47}) as long as
(\ref{2.16}), (\ref{2.17}), (\ref{2.18}) and (\ref{2.19}) hold. In fact $\phi$ is a continuous function of $v, h, e,\de$ and $ \boldsymbol{\chi}$ and Lipschitz function of  $h$, $e$ and $\boldsymbol{\chi}$ which follows from Lemma \ref{lback2} and Proposition \ref{comp0}  hence by Banach fixed point theorem we obtain (\ref{con1-1}).
\end{proof}

\section{Derivation of the reduced equations} In order to finish the proof of theorem (\ref{1}) we need to adjust the parameter in such a way that
$d(z)=m(z)=c_{j}=0.$ \be \label{z6} \int_{\R} X^{\star} k_{}
\om'_{}d \textbf{x}=0.\ee \be\label{z7} \int_{\R} X^{\star} k_{} Z d
\textbf{x}=0.\ee \be \label{z8} \int_{\R^2} \rho_{j}[N(\varphi) +
S(\textbf{w})] \om_{j, x}+ \int_{\R^2}[p\textbf{w}^{p-1}-
p\om_{j}^{p-1}]\rho_{j}\phi_{j} \om_{j, x} dx+ p
\int_{\R^2}\rho_{j}{\bf w}^{p-1}\psi_{1}\om_{j, x} =0.\ee for all
$j\in \N.$ We will call (\ref{z6}), (\ref{z7}) and (\ref{z8}) as the
{\em reduced system.}  In other words our main idea is to estimate
the lower order terms of (\ref{z6}), \ref{z7} and (\ref{z8}). We
show that (\ref{z6}) and  (\ref{z7}) is equivalent to a nonlocal
nonlinear system of second order differential equations with in
variable $h$, $e$ and $\boldsymbol{\chi}.$  From (\ref{z8}) we
obtain an infinite dimensional Toeplitz matrix. Choose $0< \mu<1.$
Define $\nu_{}= \min\{k_{1}, k_{2}, k_{3}, k_{4}, k_{5},
\frac{3}{4}\sigma\}.$

\begin{prop}\label{po1} Then (\ref{z6}) is equivalent to the following differential equation:
\be\label{redu1} c_{1}(h+ v)''- \frac{\pa \Psi_{L}}{\pa f}(h+ v)= \mathcal{P}\ee
where $\mathcal{P}$ satisfies the following inequality
\be \|\mathcal{P}\|_{C^{0, \mu}_{\theta}(\R)}\leq C \al^{1+\nu}.\ee
Moreover,  $\mathcal{P}$ satisfies Lipschitz property
\bea \|\mathcal{P}(h^{(1)}, e^{(1)}, \boldsymbol{\chi}^{(1)},.)- \mathcal{P}(h^{(2)}, e^{(2)}, \boldsymbol{\chi}^{(2)}, .)\|_{C^{0, \mu}_{\theta}(\R)} &\leq& C (\| h^{(1)}- h^{(2)}\|_{C^{2, \mu}_{\theta}(\R)}\no \\&+& \| e^{(1)}- e^{(2)}\|_{C^{2, \mu}_{\theta}(\R)}+ \al \| \boldsymbol{\chi}^{(1)}-  \boldsymbol{\chi}^{(2)}\|_{\al}).\eea
\end{prop}
\begin{proof} It is easy to check that the main term in the projection of $X_{}^{\star}k$  on $\om'$ is given by $X_{}^{\star}(\rho S({\bf w})).$ We express the laplacian in the local coordinates, using the notation of (\ref{lap1}), and neglecting the higher order terms in $\al$. Then we have \bea\label{sw4} \int_{\R}X_{}^{\star}((\rho S({\bf w}))\om')d\text{x}\sim \int_{\R} b_{1}(\pa_{\text{x}}\om)^2 d\text{x}+ \sum_{j=1}^{\infty }p\int_{\R} \om^{p-1} \om_{j} \pa_{\text{x}}\om d{\text{x}}.\eea
We will show in the later part of the proof that the difference of the left hand side and the right hand side of (\ref{sw4}) is very small in terms $\al.$
We first compute the integral using (\ref{b11}) and we obtain
\bea  \int_{\R} b_{1}(\pa_{\text{x}}\om)^2 d\text{x}&=& \int_{\R} (\pa_{\text{x}}\om)^2 \frac{1}{A^3}(-\kappa_{}A^2-h''A+(h')^2 \kappa_{}- (\text{x}+h)h' \kappa_{}) d\text{x}\no\\&=& -(f''+ h'')\int_{\R} (\om'(x))^2dx +\mathcal{O}_{C^{2, \mu}_{\theta}(\R)}(\|h\|^2_{C^{2, \mu}_{\theta}(\R)}+ \|f\|^2_{C^{3, \mu}_{\theta}(\R)}).\eea
Now we want to compute the terms involving the interaction between the spikes and the front. Let $j=1.$ In fact it is easy to note that for $j\geq 2$
the terms involved is of the higher order. Using the estimate in Section 2 we obtain
\bea\no p\int_{\R} \om^{p-1}(\text{x}) \om_{1}(x, z) \pa_{\text{x}}\om(\text{x}) d{\text{x}}&=& -\int_{\R} \om^{p}(x) \om_{1, x}(x, z) d{{x}}+ \mathcal{O}_{C^{0, \mu}_{\theta}(\R)}(\al^{1+\nu}) \\&=&\frac{\pa \Psi_{L}}{\pa f}(h+ v)+ \mathcal{O}_{C^{0, \mu}_{\theta}(\R)}(\al^{1+\nu})\no.\eea
Now we precisely calculate some of the terms involved in estimating
\be \no \int_{\R}X_{}^{\star}((\rho S({\bf w}))\om')d\text{x}.\ee
We first calculate the
\be\label{zs3} \int_{\R} a_{12}\Xi_{}\rho(\pa^{2}_{\text{x}, \text{z}}\om_{\de})\om d\text{x}\sim +\al \sqrt{\lambda_{1}}h'\de_{}\sin (\sqrt{\lambda_{1}}\text{z})\Xi_{}\int_{\R}\om Z dx\ee

Now we estimate the right hand side of (\ref{zs3}). Then we have
\be \no  \al |\de| \|\sqrt{\lambda_{1}}h'\sin \sqrt{\lambda_{1}}{\text{z}}\Xi_{}\|_{C^{0, \mu}_{\theta}(\R)}\leq C \al^{2+ k_{2}+ k_{4}}=\mathcal{O}_{C^{0, \mu}_{\theta}(\R)}(\al^{2+\nu}).\ee
>From (\ref{sw2}) we have
\be\no X_{}^{\star}(\rho(\mathcal{L}-\De+1)\psi_{2})\sim X_{}^{\star} (\rho \textbf{w}^{p-1}\psi_{2})\ee
Using (\ref{5.42}) we obtain
\be \bigg\|\int_{\R}  X_{}^{\star} (\rho \textbf{w}^{p-1}\psi_{2}) \om'd\text{x}\bigg\|_{C^{0, \mu}_{\theta}(\R)}\leq  \mathcal{O}_{C^{0, \mu}_{\theta}(\R)}(\al^{1+\frac{3}{4}\sigma}).\ee
Moreover, the last term in (\ref{sw2})
\be\no  \int_{\R}[- X^{\star}(\rho \mathcal{L}\phi)+  X_{}^{\star}\rho [\pa ^2_{\text{x}}+ \pa ^2_{\text{z}}- F'(\om)]X_{}^{\star}\phi ]\om'\sim
\int_{\R }X_{}^{\star}[\rho(f'(\text{w})- f'(\om))]\phi \om' d\text{z}.\ee
Hence we have
\be \no \| \int_{\R }X_{}^{\star}[\rho(f'(\text{w})- f'(\om))]\phi \om' d\text{z}\|_{{C^{0, \mu}_{\theta}(\R)}}\leq C |\de| \|X_{}^{\star}\phi\|_{C^{2, \mu}_{\sigma,\theta}(\R^2)}+ C \al\|X_{}^{\star}\phi\|_{C^{2, \mu}_{\sigma,\theta}(\R^2)}\leq C (\al^{2+ k_{4}}+ \al^2).\ee
It is easy to check that the other terms are of the higher  order in $\al$. Hence we obtain
\be\no c_{1}(h+ v)''-\frac{\pa \Psi_{L}}{\pa f }(h+ v)=\mathcal{O}_{C^{0, \mu}_{\theta}(\R)}(\al^{1+\nu_{}})\ee
where $c_{1}= \int_{\R} (\om'(x))^2dx.$
The continuity and the Lipschitz property of $\mathcal{P}$ can be obtained in a standard way using the estimate of the error in Proposition \ref{comp0},  the Lipschitz estimate of $\psi_{2}$ and $\phi$.
\end{proof}

\begin{prop}\label{po2} We have (\ref{z7}) is equivalent to the following differential equation:
\be\label{redu2} e''+\lambda_{1}e=\mathcal{R}\ee
where $\mathcal{R}$ satisfies the following inequality
\be \|\mathcal{R}\|_{C^{0, \mu}_{\theta}(\R)}\leq C \al^{2+\nu_{}}.\ee
Moreover,  $\mathcal{R}$ satisfies Lipschitz property
\bea \|\mathcal{R}(h^{(1)}, e^{(1)}, \boldsymbol{\chi}^{(1)}, .)- \mathcal{R}(h^{(2)}, e^{(2)}, \boldsymbol{\chi}^{(2)}, .)\|_{C^{0, \mu}_{\theta}(\R)} &\leq& C (\| h^{(1)}- h^{(2)}\|_{C^{2, \mu}_{\theta}(\R)}\no \\&+& \| e^{(1)}- e^{(2)}\|_{C^{2, \mu}_{\theta}(\R)}+ \al \| \boldsymbol{\chi}^{(1)}-  \boldsymbol{\chi}^{(2)}\|_{\al}).\eea
\end{prop}
\begin{proof} It is easy to check that the dominating term in (\ref{redu2}) is given by
\be\label{5.48} \int_{\R}X_{}^{\star}(\rho S({\bf w})X_{}^{\star} Z) d\text{x}\sim \int_{\R} [\pa_{\text{x}}^{2}+ \pa_{\text{z}}^{2}+ f(\om (\text{x}))]e(\text{z})Z(\text{x})\rho Z(\text{x})d\text{x}.\ee
But we know that \be \no \{\pa^2_{\text{x}}+ f(\om (\text{x}))\}Z= \lambda_{1}Z\ee and hence we have the right hand side of (\ref{5.48}) reduces to
\bea\label{5.49} \int_{\R} [\pa_{\text{x}}^{2}+ \pa_{\text{z}}^{2}+ f(\om (\text{x}))]e(\text{z})Z(\text{x})\rho Z(\text{x})d\text{x}&\sim& \int_{\R} (\pa_{\text{z}}^2+\lambda_{1}) e(\text{z})\rho Z^2)\no\\&\sim&  ( e''(\text{z})+ \lambda_{1} e_{})\int_{\R} \rho Z^2 dx.\eea
This gives the reduced equation for $e.$  The Lipschitz property follows in a standard way.
\end{proof}

We have from (\ref{reso})
\be\label{redu40}\Theta=\int_{\R^{+}}\int_{\R}X_{}^{\star} (k Z)Z(\text{x})\cos(\sqrt{\lambda}_{1}\text{z})d\text{x}d\text{z}=0.\ee
>From (\ref{redu40}) we deduce the reduced equation for the parameter $\de$.

\begin{lemma}\label{po3} Moreover,
\be\label{a}\Theta=\varsigma \sqrt{\lambda_{1}} \de_{}+ \mathcal{O}(\al^{1+\nu_{}}).\ee
\end{lemma}

\begin{proof} From (\ref{redu40}) we have
\be\no \Theta= \int_{\R^{+}}^{}\int_{\R}X_{}^{\star} (\rho S(\text{w}))Z(\text{x})\cos(\sqrt{\lambda}_{1}\text{z})d\text{x}d\text{z}+ \mathcal{O}(\al^{1+\nu_{}})\ee
where $\text{w}$ is defined in (\ref{ap}). But we have from
\be\label{g5} X_{}^{\star} (\rho S(\text{w}))\sim \pa^2_{\text{x}}\text{w}+ \pa^2_{\text{z}}\text{w}+ F(\text{w}).\ee
But using the fact that $\Xi_{}+\Xi_{0}=1 $ we have
\bea \pa^2_{\text{x}}\text{w}+ \pa^2_{\text{z}}\text{w}+ F(\text{w}) &\sim& [\Xi_{}'' \om_{\de}+ \Xi_{0}'' \om_{}]\no\\&+& 2 [\Xi_{}' \pa_{\text{z}}\om_{\de}+ \Xi_{0}' \pa_{\text{z}}\om_{}]\no\\&=& [\Xi_{}'' (\om_{\de}- \om_{0})]+ 2 \Xi_{}' \pa_{\text{z}}\om_{\de}\eea
Further we have
\be \pa_{\text{z}}\om_{\de} \sim -\sqrt{\lambda_{1}} Z \de_{}\sin (\sqrt{\lambda_{1}}\text{z})\ee
\be (\om_{\de}- \om_{0}) \sim Z \de_{}\cos (\sqrt{\lambda_{1}}\text{z}),\ee
where the neglected terms are of higher order $\mathcal{O}_{C^{\infty}(\R)}(|\de|^{2})(\cosh \text{x})^{-1}$ and consequently their contribution is small.
Then from (\ref{g5}) we have
\be\no \Theta \sim \varsigma \de \sqrt{\lambda_1} \int_{\R^{+}} \Xi_{}'\sin 2(\sqrt{\lambda_{1}}\text{z}) d\text{z}=\varsigma \sqrt{\lambda_{1}} \de \no\ee where $\varsigma= \int_{\R}\rho Z^2.$
\end{proof}

\begin{prop} We have (\ref{z8}) is equivalent to the following system of equations
\be \label{redu3o} \gamma_{0} (e^{-|\xi_{1}-f(0)|}\chi_{1}- e^{-|\xi_{2}-\xi_{1}|}(\chi_{2}-\chi_{1}))= \mathcal{G}_{1}(v, h, e, \de_{}, \boldsymbol{\chi})\ee
and for $j\geq 2$  we have
\be\label{redu3} \gamma_{0} (e^{-|\xi_{j+1}-\xi_{j}|}(\chi_{j+1}-\chi_{j})- e^{-|\xi_{j}-\xi_{j-1}|})(\chi_{j}-\chi_{j-1}))= \mathcal{G}_{j}(v, h, e, \de_{}, \boldsymbol{\chi})\ee
where $\mathcal{G}= \{\mathcal{G}_{j}\}_{j\geq 1}$ satisfies the following inequality
\be \|\mathcal{G}\|_{\al}=  \max_{i} \al^{-i}|\mathcal{G}_{i}|\leq C \al^{1+\nu}.\ee
Moreover,  $\mathcal{G}$ satisfies Lipschitz property
\bea \|\mathcal{G}(h^{(1)}, e^{(1)}, \boldsymbol{\chi}^{(1)}, .)- \mathcal{G}(h^{(2)}, e^{(2)}, \boldsymbol{\chi}^{(2)}, .)\|_{\al} &\leq& C  (\| h^{(1)}- h^{(2)}\|_{C^{0, \mu}_{\theta}(\R)}\no \\&+& \| e^{(1)}- e^{(2)}\|_{C^{0, \mu}_{\theta}(\R)}+ \al \| \boldsymbol{\chi}^{(1)}-  \boldsymbol{\chi}^{(2)}\|_{\al}).\eea
and continuous in the remaining variables.
\end{prop}
\begin{proof} Without loss of generality let $\gamma_{0}= \gamma_{1}$.
Using the estimates (\ref{ke1}) and (\ref{ke3}) we obtain (\ref{redu3o}) and (\ref{redu3}). Now we estimate some of the terms involved in $\mathcal{G}.$
\be \label{mn1} p\int_{\R^2} {\bf w}^{p-1}\psi_{1}\om_{j, x}dx dz= \mathcal{O}(\al^{1+\frac{3}{4}\sigma+ j}).\ee
and
\be \label{mn2}\int_{\R^2} N(\varphi)\om_{j, x}dxdz\leq C\int_{\R^2} |\varphi|^2\om_{j, x}dxdz=
 \mathcal{O}(\al^{2+j}).\ee
Now we precisely calculate some of the terms involved in estimating (\ref{redu3o})-(\ref{redu3})
\be\label{mn3}\int_{\R^2}X_{}^{\star}((\rho S({\bar{w}}))\om_{j, \text{x}})d\text{x} d\text{z}\sim \int_{\R^2}X_{}^{\star}((\rho S({\text{w}})))\om_{j, \text{x}}d\text{x} d\text{z}\ee thus neglecting the higher order term in $\al.$
We first calculate the lower order term in the expression
\be\no \int_{\R^2} |a_{2}\Xi_{}\rho(\pa^{2}_{\text{x}, \text{z}}\om_{\de})\om_{j, \text{x}}|d\text{x} d\text{z}\leq \sqrt{\lambda_{1}} \al^{1+ k_{2}}|\de|\int_{\R^2}|\om_{j, x}| dxdz= \mathcal{O}(\al^{2+ k_{2}+ k_{4}+j}).\ee
\end{proof}

\section{Solution of the reduced systems and proof of Theorem \ref{1.1}}

\subsection{Proof of Theorem \ref{1.1}} We now complete the proof of Theorem \ref{1}. To this end we have to solve the following system of equations
\be\label{line1} c_{1}(h+ v)''-\frac{ \Psi_{L}(f, \text{z})}{\pa f}(h+ v)= \mathcal{P}(v, h, e, \de , \boldsymbol{\chi})\ee
\be\label{line2} e''+ \lambda_1 e=\mathcal{R}(v, h, e, \de, \boldsymbol{\chi})\ee
\be \label{line3} \sqrt{\lambda}_{1}\varsigma_{0} \de = \Theta(v, h, e, \de_{},\boldsymbol{\chi})\ee
\begin{equation}
  \label{line4}
 \left\{\begin{aligned}\gamma_{0} (e^{-|\xi_{1}-f(0)|}\chi_{1}- e^{-|\xi_{2}-\xi_{1}|}(\chi_{2}-\chi_{1}))&= \mathcal{G}_{1}(v, h, e,  \de, \boldsymbol{\chi})\\
\gamma_{0} (e^{-|\xi_{j+1}-\xi_{j}|}(\chi_{j+1}-\chi_{j})- e^{-|\xi_{j}-\xi_{j-1}|}(\chi_{j}-\chi_{j-1}))&= \mathcal{G}_{j}(v, h, e,  \de, \boldsymbol{\chi}).
 \end{aligned}
  \right.
\end{equation}

\begin{prop} The system (\ref{line1})-(\ref{line4}) is a one parameter family of solutions in the sense that for each choice of $\de \in \R,$ the system admits a solution containing $\de$ and the functions $v, h, e$ and the parameter $\boldsymbol{\chi}.$
\end{prop}
\begin{proof} First we choose $k_{i}\in (0, 1), \mu \in (0,1)$ and $0< \sigma< \min\{p-2, 1\}$ in such a way that
\be \nu_{}=\min\{k_{1}, k_{2}, k_{3}, k_{4}, k_{5},
\frac{3}{4}\sigma\}.\ee Fix $\de$ and moreover assume that the
parameter satisfy \be |\de | \leq \frac{1}{2} \al^{1+ k_{4}}.\ee
In order to complete the proof we need to go through the following steps.\\
$\bullet$ Firstly we define $\tilde{v}, \tilde{h}, \tilde{e}, \bar{\de}, \bar{\boldsymbol{\chi}}$. We define $\tilde{\de}= \bar{\de}+ \de_{}$  and use this parameter $\de$ to calculate the right hand sides of (\ref{line1})-(\ref{line4}). Then these functions satisfy the assertions  of Propositions \ref{po1}, \ref{po2} and Lemma \ref{po3}. In particular, they are Lipschitz functions of $\tilde{h}$, $\tilde{e}$ and $\boldsymbol{\chi};$ and continuous functions of $\tilde{v}$ and $\de.$\\
$\bullet$ We now apply Banach fixed point theorem to the solve (\ref{line1})-(\ref{line4}) for $h$,$e$ and $\boldsymbol{\chi}$. Also we note that
\be \no \|h\|_{C^{0, \mu}_{\theta}(\R)}\leq C \|\mathcal{P}\|_{C^{0, \mu}_{\theta}(\R)}\leq C \al^{1+\nu_{}}\ee
\be \no \|e\|_{C^{0, \mu}_{\theta}(\R)}\leq C \|\mathcal{R}\|_{C^{0, \mu}_{\theta}(\R)}\leq C \al^{2+\nu_{}}\ee and
it is easy to check that
\be \no \|\boldsymbol{\chi}\|_{\al}\leq C \al^{-1}\|\mathcal{G}\|_{\al}\leq C \al^{\nu}\ee
and $v$, $\de_{}$  satisfy
\be \no  \|v\|_{\mathcal{E}}\leq C \al^{1+\nu_{}}\ee
\be \no  |\de|\leq C \al^{1+ \nu_{}}.\ee
$\bullet$ Now we define a continuous map on a finite dimensional space $\mathcal{E}\times \R$ $$\mathcal{F}:\mathcal{E}\times \R\rightarrow \mathcal{E}\times \R $$ given by $$(\bar{v}, \bar{\de})\mapsto  (v, \de ).$$
By the choice of $\nu,$ we can use Browder's fixed point theorem to obtain a fixed point of the map $\mathcal{F}.$\\
\end{proof}

\subsection{Final remarks on the proofs of Solution 2 and Solution 3}

Finally, we show what modifications are needed for the proofs of Solution 2 and Solution 3 in Section 1.2.

For Solution 2, we  use approximate solution of the following form
\begin{equation}
  \label{a4-100}
\left\{\begin{aligned} u_{L}(x, z)&= u_{L}(x, -z) &&\text{for all } (x,z)\in \R^2 \\ u_{L}(x,z)&= \bigg(\om_{\delta} (x-f(z)-h_L(z),z)-
 &&\sum_{i=1}^{\infty}\om_{0}((x, z)+\xi_i\vec{e}_{1})\bigg)(1+o_{L}(1))\end{aligned}
  \right.
\end{equation}
where the interaction function  $f$ satisfies
\begin{equation}
  \label{a2-100}
\left\{\begin{aligned}
       f''(z) &= - \Psi_L (f, z)  &&\text{in } \R\\
      f(0) &= 0, \ \       f'(0) = 0.
    \end{aligned}
  \right.
\end{equation}

For Solution 3, we  consider $\xi_{1}=0$ and use the  approximate solution of the following form
\begin{equation}
  \label{a4-101}
\left\{\begin{aligned}
u_{L}(x, z)&= u_{L}(x, -z) &&\text{for all } (x,z)\in \R^2 \\
u_{L}(x, z)&= u_{L}(-x, z) &&\text{for all } (x,z)\in \R^2 \\
u_{L}(x,z)&= \bigg(\om_{\delta} (x-f(z)-h_L(z),z) +\omega_{\delta} (x+f(z)+h_L(z), z)
 && + \om_{0}(x, z)\bigg)(1+o_{L}))\end{aligned}
  \right.
\end{equation}
where   $f$ satisfies (\ref{a2}).

The rest of  the proofs remains the same.

\bigskip
\section*{acknowledgement} The first author was supported by an ARC grant DP0984807 and the second author was supported from a General Research Fund
from RGC of Hong Kong, Joint Overseas Grant of NSFC,  and a Focused Research Scheme of CUHK. We are indebted to Prof. F. Pacard for
suggesting this problem and for many useful conversations.  We also
thank Prof. M. Kowalczyk and Prof. M. del Pino for many
constructive conversations.


\begin{thebibliography}{99}
\bibitem{A} {\sc A. Aleksandrov;} Uniqueness theorems for surfaces in the large. I. {\em Amer. Math. Soc. Transl.}  21 (1962) 341--354.

\bibitem{END} {\sc E. N. Dancer;}  New solutions of equations on $\R^N$. {\em Ann. Scuola Norm. Sup. Pisa Cl. Sci.} (4) 30 (2001), no. 3-4, 535--563 (2002).

\bibitem{CD} {\sc C. Delaunay;} Sur les surfaces de r\'evolution dont la courbure moyenne est constante. (French) {\em Jounal de math\'ematiques } 26 (1898), 43--52.

\bibitem{DKPW1}  {\sc M. del Pino, M. Kowalcyzk, F. Pacard, J. Wei;} The Toda system and multiple-end solutions of autonomous planar elliptic problems. {\em Advances in Mathematics}  224 (2010), 1462--1516.

\bibitem{DKPW2}  {\sc M. del Pino, M. Kowalcyzk, F. Pacard, J. Wei;} Multiple-end solutions to the Allen-Cahn equations in $\R^2.$   {\em J. Funct.  Anal. } 258
(2010),458--503

\bibitem{DKW1}  {\sc M. del Pino, M. Kowalcyzk, J. Wei;} Concentration on curves for nonlinear Schr\"odinger equations. {\em Comm. Pure Appl. Math.} 60 (2007), no. 1, 113--146.

\bibitem{DKW2}  {\sc M. del Pino, M. Kowalcyzk, J. Wei;} The Toda system and clustering interfaces in the
 Allen-Cahn equation. {\em Arch. Ration. Mech. Anal.} 190 (2008), no. 1, 141--187.

\bibitem{FW} {\sc A. Floer, A. Weinstein;} Nonspreading wave packets for the cubic Schr\"{o}dinger equation with a bounded potential. {\it J. Funct. Anal.} {\bf 69} (1986), no. 3, 397--408.


\bibitem{G} {\sc K. G-Brauckmann;} New surfaces of constant mean curvature. {\em Math. Z.} 214 (1993), no. 4,
527--565.

\bibitem{GKK} {\sc K. G-Brauckmann, R. B. Kusner; M. Sullivan;} Triunduloids: embedded constant mean curvature surfaces with three ends and genus zero. {\it J. Reine Angew. Math.} 564 (2003), 35--61.

\bibitem{GKK2} {\sc K. G-Brauckmann, R. B. Kusner; M. Sullivan;} Constant mean curvature surfaces with three ends. {\em Proc. Natl. Acad. Sci. USA} 97 (2000), no. 26, 14067--14068

\bibitem{GK} {\sc K. G-Brauckmann, R. B. Kusner;} Embedded constant mean curvature surfaces with special symmetry. {\em Manuscripta Math. } 99 (1999), no. 1, 135--150.

\bibitem{GNN} {\sc B. Gidas, W. M. Ni, L. Nirenberg; }  Symmetry of positive
solutions of nonlinear elliptic equations in $R^{N}$, {\em Adv. Math.
Suppl. Stud.} 7A (1981) 369--402.

\bibitem{Kap} {\sc N. Kapouleas;} Complete constant mean curvature surfaces in Euclidean three-space. {\em Ann. of Math.} (2) 131 (1990), no. 2, 239--330.

\bibitem{KKS} {\sc N. Korevaar, R. Kusner, B. Solomon;} The structure of complete embedded surfaces with
constant mean curvature. {\em J. Differential Geom.} 30 (1989), no. 2, 465--503.

\bibitem{K} {\sc Man Kam Kwong;} Uniqueness of positive solutions of $\De u-u+ u^{p}=0.$ {\em Arch. Ration. Mech. Anal.} 105 (1989), no. 5, 243--266.

\bibitem{LWW} {\sc F. H. Lin,  W. M. Ni, J. Wei;} On the number of interior peak solutions for a singularly perturbed Neumann problem. {\it Comm. Pure Appl. Math.} 60 (2007), no. 2, 252--281.

\bibitem{MM1} {\sc A. Malchiodi, M. Montenegro;}
 Boundary concentration phenomena for a singularly perturbed elliptic problem,
{\em  Commun. Pure Appl. Math.} 55 (2002), 1507-1568.

\bibitem{MM2} {\sc A. Malchiodi, M. Montenegro;}
Multidimensional boundary layers for a singularly perturbed Neumann problem.
{\em Duke Math. J. } 124 (2004), no. 1, 105-143.

\bibitem{MM}  {\sc F. Mahmoudi, A. Malchiodi;} Concentration on minimal submanifolds for a singularly perturbed Neumann problem. {\em Adv. Math.} 209 (2007), no. 2, 460--525.

\bibitem{M} {\sc A. Malchiodi;} Some new entire solutions of a semilinear elliptic problem in $\R^N$. {\em Advances in Math. } 221 (2009), 1843-1909.

\bibitem{MP} {\sc R.Mazzeo, F. Pacard;} Constant mean curvature surfaces with Delaunay ends. {\em Comm. Anal. Geom.} 9 (2001), no. 1, 169--237.

\bibitem{Me} {\sc W. Meeks;} The topology and geometry of embedded surfaces of constant mean curvature. {\em J. Differential Geom.} 27 (1988), no. 3, 539--552.

\bibitem{Nisurvey} {\sc W.-M. Ni; }  Qualitative properties of solutions to elliptic problems,  {\em Stationary Partial Differential Equations, vol. I, Handb. Differ. Equ. North-Holland, Amsterdam } (2004), 157-233.

\bibitem{MPW} {\sc M. Musso, F. Pacard, J. Wei;} Finite-energy sign-changing solutions with dihedral symmetry for the stationary nonlinear Schrodinger equation. {\em J. Eur. Math. Soc.}, to appear.

\bibitem{NW} {\sc W. M. Ni, Juncheng Wei;} On the location and profile of spike-layer solutions to singularly perturbed semilinear Dirichlet problems. {\em Comm. Pure Appl. Math.} 48 (1995), no. 7, 731--768.

\bibitem{OG} {\sc Yong-Geun Oh;} On positive multi-bump bound states of nonlinear Schr\"{o}dinger equations under multiple well potential. {\it Comm. Math. Phys.} {\bf 131} (1990), no. 2, 223--253.

\bibitem{OG1} {\sc Yong-Geun Oh;} Existence of semiclassical bound states of nonlinear Schr\"{o}dinger equations with potentials of the class $(V)\sb a$. {\it Comm. PDE} {\bf 13} (1988), no. 12, 1499--1519.



\bibitem{PQS} {\sc P. Polacik, P. Quittner, P. Souplet;}  Singularity and decay estimates in superlinear problems via Liouville-type theorems. I. Elliptic equations and systems. {\em Duke Math. J.} 139 (2007), no. 3, 555--579.

\bibitem{ratzin1} {\sc J. Ratzkin;} An end-to-end gluing construction for surfaces of constant mean curvature, PhD Thesis, University of Washington (2001).

\bibitem{ratzin2} {\sc J. Ratzkin;} An end-to-end gluing construction for metrics of constant positive scalar curvature, {\em Indiana Univ. Math. J.} 52(2003), 703-726.

\bibitem{Wei} {\sc J.-C. Wei; }  Existence and Stability of Spikes for the Gierer-Meinhardt System,  {\em Stationary Partial Differential Equations, vol. V, Handb. Differ. Equ. North-Holland, Amsterdam} (2008), 487-585.

\end{thebibliography}
\end{document}